\documentclass[12pt]{article}

\usepackage{amsmath}
\usepackage{amsfonts}
\usepackage{latexsym}
\usepackage{amssymb}
\usepackage{curves}
\usepackage{graphicx}
\usepackage{wasysym}
\usepackage{MnSymbol}
\usepackage{amsthm,amscd}

\newtheorem{thm}{Theorem}[section]
\newtheorem{pro}[thm]{Proposition}
\newtheorem{cor}[thm]{Corollary}
\newtheorem{lem}[thm]{Lemma}

\newtheorem{example}[thm]{Example}
\newtheorem{rem}[thm]{Remark}

\newcommand{\mapsfrom}
{\mathrel{\reflectbox{\ensuremath{\mapsto}}}}

\newcommand{\comp}{\, {}_\circ \,}

\newcommand{\noin}{\noindent}

\newcommand{\dozspace}{\;\;\;\;\;\;\;\;\;\;\;\;}

\newcommand{\kay}{{\mathbb{K}}}

\newcommand{\FF}{{\mathbb{F}}}

\newcommand{\NN}{{\mathbb{N}}}
\newcommand{\OO}{{\mathbb{O}}}

\newcommand{\QQ}{{\mathbb{Q}}}

\newcommand{\ZZ}{{\mathbb{Z}}}

\newcommand{\openbin}{\left( \!\! \begin{array}{c}}
\newcommand{\closebin}{\end{array} \!\! \right)}
\newcommand{\openvec}{\left[ \!\! \begin{array}{c}}
\newcommand{\closevec}{\end{array} \!\! \right]}
\newcommand{\openmat}{\left[ \!\! \begin{array}{cc}}
\newcommand{\closemat}{\end{array} \!\! \right]}
\newcommand{\opentri}{\left[ \!\! \begin{array}{ccc}}
\newcommand{\closetri}{\end{array} \!\! \right]}
\newcommand{\openquad}{\left[ \!\! \begin{array}{cccc}}
\newcommand{\closequad}{\end{array} \!\! \right]}
\newcommand{\openquin}{\left[ \!\! \begin{array}{ccccc}}
\newcommand{\closequin}{\end{array} \!\! \right]}
\newcommand{\opensex}{\left[ \!\! \begin{array}{cccccc}}
\newcommand{\closesex}{\end{array} \!\! \right]}

\newcommand{\ltri}{\medtriangleleft}
\newcommand{\rtri}{\medtriangleright}

\newcommand{\dash}{\mbox{\rm{--}}}

\newcommand{\Aut}{{\rm{Aut}}}

\newcommand{\End}{\mbox{\rm{End}}}

\newcommand{\Ind}{\mbox{\rm{Ind}}}

\newcommand{\Iso}{\mbox{\rm{Iso}}}

\newcommand{\Mat}{\mbox{\rm{Mat}}}

\newcommand{\Res}{\mbox{\rm{Res}}}

\newcommand{\ad}{{\rm{ad}}}

\newcommand{\epi}{{\rm{epi}}}

\newcommand{\id}{{\rm{id}}}

\newcommand{\len}{\mbox{\rm{len}}}

\newcommand{\mob}{\mbox{\rm{m\"{o}b}}}
\renewcommand{\mod}{\mbox{\rm{mod}}}
\newcommand{\mor}{\mbox{\rm{mor}}}
\newcommand{\obj}{{\rm{obj}}}

\newcommand{\tr}{\mbox{\rm{tr}}}

\newcommand{\oo}{\overline}
\newcommand{\uu}{\underline}

\newcommand{\frakP}{{\mathfrak P}}

\newcommand{\frako}{{\mathfrak o}}

\newcommand{\cA}{{\cal A}}
\newcommand{\cB}{{\cal B}}
\newcommand{\cC}{{\cal C}}

\newcommand{\cE}{{\cal E}}

\newcommand{\cK}{{\cal K}}
\newcommand{\cL}{{\cal L}}
\newcommand{\cM}{{\cal M}}

\newcommand{\cO}{{\cal O}}
\newcommand{\cP}{{\cal P}}

\newcommand{\cR}{{\cal R}}
\newcommand{\cS}{{\cal S}}
\newcommand{\cT}{{\cal T}}

\begin{document}


\title{Semisimplicity of some deformations of
the subgroup category and the biset category}

\author{\large Laurence
Barker\footnote{e-mail: barker@fen.bilkent.edu.tr,
Some of this work was done while this author was
on sabbatical leave, visiting the Department of
Mathematics at City, University of London.}
\hspace{1in} \.{I}smail Alperen
\"{O}\u{g}\"{u}t\footnote{e-mail: ismail.ogut@bilkent.edu.tr} \\
\mbox{} \\
Department of Mathematics \\
Bilkent University \\
06800 Bilkent, Ankara, Turkey \\
\mbox{}}

\maketitle

\small
\begin{abstract}
\noin We introduce some deformations of
the biset category and prove a
semisimplicity property. We also
consider another group category, called
the subgroup category, whose morphisms
are subgroups of direct products, the
composition being star product. For some
deformations of the subgroup category,
too, we prove a semisimplicity property.
The method is to embed the deformations
of the biset category into the more easily
described deformations of the subgroup
group category.

\smallskip
\noin 2010 {\it Mathematics Subject Classification:}
Primary 19A22, Secondary 16B50.

\smallskip
\noin {\it Keywords:} twisted category algebra,
semisimple category, semisimple deformation,
seed, vanishing problem.
\end{abstract}

\section{Introduction}

This paper concerns two categories and
some of their deformations. One of those
two can be defined immediately without
any specialist prerequisites. We define the
{\bf subgroup category} $\cS$ as follows.
The class of objects of $\cS$ is the class
of groups. Consider groups $R$, $S$, $T$.
Let us write the elements of the direct
product $R {\times} S$ in the form
$r {\times} s$ instead of $(r, s)$. We
define the set of $\cS$-morphisms
$R \leftarrow S$ to be the set of
subgroups of $R {\times} S$. Given
subgroups $U \leq R {\times} S$ and
$V \leq S {\times} T$, we define the
$\cS$-composite of $U$ and $V$,
denoted $U * V$, to be the subgroup of
$R {\times} T$ consisting of those elements
$r {\times} t$ such that there exists
$s \in S$ satisfying $r {\times} s \in U$
and $s {\times} t \in V$. The operation
$*$, called the {\bf star product}, is
familiar in the theory of bisets, as in
Bouc [Bou10, 2.3.19], for instance.

We mention that a category constructed
on similar lines, with finite sets as objects,
has been shown to admit rich theory
by Bouc--Th\'{e}venaz [BT] and citations
therein.

Preliminaries aside, the groups under
consideration will be finite. We let $\cK$
be any non-empty set of finite groups.
For some of the categories we shall be
considering, the objects are arbitrary
finite groups. To preempt logical
objections when applying ring theoretic
techniques, we shall often be passing
from a large category to a small full
subcategory whose set of objects is
$\cK$. However, up to equivalence of
categories, the dependence on $\cK$
will be only on the group isomorphism
classes occurring in $\cK$. So our
rigorous constraint on expression will
impose little constraint on interpretation.

We let $\kay$ be a field of characteristic
zero. The role of $\kay$ will be as a
coefficient ring. Some easy generalizations
to other coefficient rings do not seem to
merit tergiversation over hypotheses.

The scenario under investigation will be
determined by $\cK$ and $\kay$, together
with one other item of data, namely, a
monoid homomorphism $\ell$ to the unit
group $\kay^\times = \kay - \{ 0 \}$
from the monoid of positive integers
$\NN - \{ 0 \}$. Let $\Pi$ denote the set
of primes. Then $\ell$ is determined by
the family $(\ell(q) : q \in \Pi)$ of arbitrary
elements $\ell(q) \in \kay^\times$. We
call $\ell$ {\bf algebraically independent}
with respect to $\cK$ when, letting $q$
run over the prime divisors of the orders
of the groups in $\cK$, the $\ell(q)$ are
algebraically independent over the
minimal subfield $\QQ$ of $\kay$.
The role of $\ell$ will be to parameterize
some deformations of some categories.

Let $\cS_\cK$ denote the full subcategory
of $\cS$ such that the objects of $\cS_\cK$
are the elements of $\cK$. In Section 2,
we shall construct an algebra $\Lambda$
over $\kay$. The morphisms in $\cS_\cK$
index a $\kay$-basis for $\Lambda$,
called the square basis. In terms of the
square basis, the multiplication operation
derives from the composition operation
in the most straightforward imaginable
way, except for the introduction of a
cocycle that depends on $\ell$. Thus,
employing some terminology that will
be defined in that section, $\Lambda$
is a twisted category algebra of
$\cS_\cK$.

We can view $\Lambda$ as a deformation
of the category algebra $\kay \cS_\cK$
of $\cS_\cK$. Our reason for mentioning
the less specific notion of a
deformation is that, in Section 9, we
shall be making use of $\Lambda$ to
examine an algebra $\Gamma$ that
can be viewed as a deformation of the
category $\kay \cB_\cK$ obtained
from the biset category $\cB$ by
extending coefficients to $\kay$ and
confining the objects to $\cK$. (To
prevent misunderstanding, we mention
that, in the notation of Section 2, the
category algebra of $\cB_\cK$ over
$\kay$ could also be written as
$\kay \cB_\cK$, but we shall not be
considering that category algebra
in this paper.)

Corollary \ref{6.5} says that local
semisimplicity never holds for the category
algebra $\kay \cS_\cK$, except in trivial
cases. Nevertheless, our main result,
proved in Section 7, is that $\Lambda$ is
semisimple in the following generic sense.

\begin{thm}
\label{1.1}
If $\ell$ is algebraically independent
with respect to $\cK$, then
$\Lambda$ is locally semisimple.
\end{thm}

Surprisingly, for arbitrary $\cK$, $\kay$,
$\ell$, it is not hard to classify the
simple $\Lambda$-modules. That is
in contrast to the situation for
$\kay B_\cK$, where a suitable closure
condition has to be imposed on $\cK$
to avoid the ``vanishing problem'' 
discussed in Rognerud \cite{Rog19}.
In Section 3, we shall show that, for
$\Lambda$, the ``vanishing problem''
itself vanishes. That will allow us, in
Theorems \ref{3.5} and \ref{3.6}, to
give two descriptions of a classification
of simple $\Lambda$-modules for
arbitrary $\cK$.

When $\cK$ is finite, $\Lambda$ is
unital. For unital rings, local
semisimplicity is just semisimplicity.
In Section 4, we shall prove the
following implication of that condition.
Given a positive integer $n$, we write
$\Mat_n$ to indicate an algebra of
$n {\times} n$ matrices. Given a
group $R$, we write $\Aut(R)$ to
denote the automorphism group
of $R$ in the category of groups.

\begin{thm}
\label{1.2}
Suppose $\cK$ is finite and
$\Lambda$ is semisimple. Let $E$
run over representatives of the
isomorphism classes of the factor
groups of the elements of $\cK$. Then
we have an algebra isomorphism
$$\Lambda \cong \bigoplus_E
  \Mat_{n_E}(\kay \Aut(E))$$
where $n_E$ is the number of triples
$(G, B, Y)$ such that $\cK \ni G \geq
B \unrhd Y$ and $B/Y \cong E$.
\end{thm}

In Section 5, to prepare for a deeper
study, we review some results and
techniques from Boltje--Danz
[BD13]. We express the material as
a passage to another $\kay$-basis
for $\Lambda$, the round basis,
which lacks the closure property
of the square basis but instead has
the advantage that products vanish
except under strong conditions.
After applying those techniques
to the general case in Sections 6
and 7, we shall consider, in
Section 8, a particular simple
$\Lambda$-module called the
trivial $\Lambda$-module. We
shall give criteria for the projectivity
of the trivial $\Lambda$-module.
We shall see that, for finite $\cK$,
projectivity of the trivial
$\Lambda$-module is equivalent
to simplicity of the associated
block algebra.

Theorem \ref{9.3} describes an
embedding of $\Gamma$ in
$\Lambda$. That will yield the
following corollary.

\begin{cor}
\label{1.3}
If $\Lambda$ is locally semisimple,
then $\Gamma$ is locally semisimple.
\end{cor}

For one particular case of $\ell$, we
have $\Gamma \cong \kay \cB_\cK$.
A theorem of Serge Bouc, appearing in
\cite[1.1]{Bar08}, asserts that
$\kay \cB_\cK$ is locally semisimple
if and only if every group in $\cK$ is
cyclic. In Section 9, we shall prove
the following generalization of half
of that result.

\begin{thm}
\label{1.4}
If every element of $\cK$ is cyclic, then
the algebra $\Lambda \cong \Gamma$
is locally semisimple.
\end{thm}

The aim of this paper is to initiate study
in a speculative direction which, to
yield applications, may require
generalizations. In \cite{BO}, the
$\kay$-linear biset category $\kay \cB$
is replaced by the $\kay$-linear
$C$-fibred biset category $\kay \cB_C$,
where $C$ is a supercyclic group.
Replacement of $\kay \cB$ with the
$p$-permutation category $\kay \cT$,
studied in Ducellier \cite{Duc16}, might
be difficult but interesting, since $p$-blocks
that are $p$-permutation equivalent in
the sense of Boltje-Xu [BX08] correspond
to associate idempotents of small full
subcategories of $\kay \cT$.

One possible application of semisimple
deformations of $\kay \cB$ may be in
the study of those functors $\kay \dash
\mod \leftarrow \kay \cB$ that admit
suitable deformations. The same may
pertain to $\kay \cB_C$ and $\kay \cT$,
if those two categories can be shown to
admit deformations with semisimplicity
properties.

To indicate another possible line of
further study, let us suppose that $\kay$
arises as the field of fractions of a
complete local noetherian ring
$\OO$ whose residue field $\FF$ has
prime characteristic $p$. Given an
algebra $A$ over $\OO$ such that
$A$ is freely and finitely generated
over $\OO$ and the $\kay$-linear
extension of $A$ is semisimple, then
the $\FF$-linear reduction of $A$
admits a theory of decomposition\
numbers and a factorization of the
Cartan matrix. A paradigmatic case
is that where $A$ is the group algebra
of a finite group. Another case, concerning
Mackey categories, is discussed in 
Th\'{e}venaz--Webb [TW95]. It might
seem absurd to suggest that such a
decomposition theory might be applied
in contexts involving $\kay \cB$,
$\kay \cB_C$, $\kay \cT$. After all,
those three categories lack the
prerequisite semisimplicity property.
But the suggestion may cease to seem
absurd when we consider the
possibility of reinstating semisimplicity
by passing to suitable deformations.

\section{Cocycle deformation of the
subgroup category}

After setting up some notation and
terminology, we shall define the algebra
$\Lambda$ mentioned above, and we shall
classify the simple $\Lambda$-modules.

We do not require rings to be unital.
Even when working with unital rings,
we do not require subrings to be
common-unity subrings. Given a ring
$A$, we define a {\bf corner subring}
of $A$ to be a subring $B$ such that
$B \geq BAB$. We call a ring
monomorphism $\nu : A \leftarrow C$
a {\bf corner embedding} provided
$\nu(C)$ is a corner subring. We call
$A$ {\bf locally unital} provided every
finite subset of $A$ is contained in
a unital corner subring of $A$, we
mean, a subring having the form
$e A e$ where $e$ is an idempotent
of $A$. We consider $A$-modules
only when $A$ is locally unital and,
in that case, we require that every
element of an $A$-module is fixed
by an idempotent of $A$. The next
result follows easily from the
special case in Green [Gre07, 6.2g].

\begin{pro}
\label{2.1} {\rm (Green.)} Let $B$ be
a corner subring of a locally unital
ring $A$. Then $B$ is locally unital.
Furthermore, the condition $T \cong
BS$ characterizes a bijective
correspondence $[T] \leftrightarrow [S]$
between the isomorphism classes $[T]$
of simple $B$-modules and the
isomorphism classes of those simple
$A$-modules $S$ satisfying
$BS \neq 0$.
\end{pro}

For any property $\frakP$ of unital
rings such that $\frakP$ is closed
under passage to corner subrings,
a ring $A$ is said to be {\bf locally
$\frakP$} provided $A$ is locally unital
and $\frakP$ holds for every unital
corner subring of $A$. We shall be
especially concerned with the
condition of local semisimplicity.
(The common practice of using
{\it semisimple} to mean {\it locally
semisimple} may be harmless,
since it does not change the
meaning of {\it semisimple} in the
established context of unital rings.
Nevertheless, we adopt the longer
term because it carries a cautionary
reminder of the generality.) Before
we depart from abstract ring theory,
let us record a lemma for later use.

\begin{lem}
\label{2.2}
Let $A$ be a locally artinian
ring. Let $B$ be a corner subring of
$A$. Let $S$ be a simple $A$-module.
Define $T = BS$, which is a simple
$B$-module by the above proposition.
Then $B$ is locally artinian and we have
an isomorphism of division rings
$\End_B(T) \cong \End_A(S)$.
\end{lem}

\begin{proof}
Plainly, $B$ is locally artinian.
The specified action plainly yields
an embedding of division rings
$\nu : \End_B(T) \leftarrow
\End_A(S)$. Let $i$ be a primitive
idempotent of $B$ such that
$iT \neq 0$. Then $Bi$ and $Ai$
are projective covers of $T$ and
$S$, respectively. So $\End_B(T)
\cong iBi/J(iBi)$ and $\End_A(S)
\cong iAi/J(iAi)$. But $iBi = iAi$.
\end{proof}

We understand an {\bf algebra} over
$\kay$ to be a ring equipped with a
compatible $\kay$-module structure.
That is to say, algebras over $\kay$
are to be associative but not necesarily
unital.

We deem all categories to be locally small.
Given objects $X$ and $Y$ of a category
$\cC$, we write $\cC(X, Y)$ to denote the
set of $\cC$-morphisms $X \leftarrow Y$.
It is to be understood that a morphism
determines its domain and codomain, in
other words, the morphism sets $\cC(X, Y)$
are mutually disjoint. We write $\id_X^\cC$
or $\id_X$ to denote the identity
$\cC$-morphism of $X$. When $\cC$ is
small, we write $\mor(\cC)$ for the set of
morphisms of $\cC$, we write $\obj(\cC)$
for the set of objects of $\cC$, and we call
$\cC$ a category on $\obj(\cC)$. For
arbitrary $\cC$ and a set $\cO$ of objects
of $\cC$, we write $\cC_\cO$ for the full
subcategory of $\cC$ on $\cO$.

Recall, a category is said to be
{\bf $\kay$-linear} when the morphism
sets are $\kay$-modules and the
composition maps are $\kay$-bilinear.
When $\cC$ is small and $\kay$-linear,
we define the {\bf algebra
associated with $\cC$} to be the algebra
$$\cC_{\rm alg} = \bigoplus_{X, Y \in
  \obj(\cC)} \cC(X, Y)$$
with multiplication given by composition,
the product of two incompatible morphisms
being zero. Systematically, in this paper,
we shall employ the language of category
theory when working with $\kay$-linear
categories that are possibly large, but
we shall shift to the richer language of
ring theory when working with small
$\kay$-linear categories. When $\cC$
is small and $\kay$-linear, all the
features of $\cC$ can be recovered
from the algebra $\cC_{\rm alg}$
together with the complete family of
mutually orthogonal idempotents
$(\id_X^\cC : X \in \obj(\cC))$. For
instance, the morphism $\kay$-modules
can be recovered from the equality
$\cC(X, Y) = \id_X^\cC . \cC .
\id_Y^\cC$. We shall write $\cC$
instead of $\cC_{\rm alg}$, relying on
context to resolve any ambiguity. To
diminish or eliminate even any potential
for ambiguity, we shall work freely with
the following alternative definition which,
at least in the context of our ring theoretic
approach, is equivalent to the definition
above: a small $\kay$-linear category is
an algebra $\cC$ over $\kay$ equipped
with a complete family of mutually
orthogonal idempotents whose indexing
set, denoted $\obj(\cC)$, is called the
set of objects of $\cC$.

For small $\cC$, given a subset
$\cO \subseteq \obj(\cC)$, then
$\cC_\cO$ is a corner subalgebra of
$\cC$. Note that $\cC_\cO$ is unital
if and only if $\cO$ is finite. Proof of the
next remark is easy.

\begin{rem}
\label{2.3}
Let $\cC$ be a small $\kay$-linear
category. Then the following three
conditions are equivalent: $\cC$ is
locally semisimple; every full subcategory
of $\cC$ is locally semisimple; $\cC_\cO$
is semisimple for every finite set $\cO$
of objects of $\cC$.
\end{rem}

Given $\cC$ as in the remark, a
$\cC$-module $M$ and $X \in \obj(\cC)$,
we define $M(X) = \id_X . M$, which we
regard as a module of the endomorphism
algebra $\End_\cC(X) = \cC(X, X) =
\id_X^\cC . \cC . \id_X^\cC$. We mention
that, in a well-known manner, $M$ can be
viewed as a functor to the category of
$\kay$-modules, and $M(X)$ can be
regarded as the evaluation at $X$. But we
shall not be making use of that
interpretation.

As another preliminary, let us say a few
words on category algebras and twisted
category algebras. The following
constructions are already discussed in
Linckelmann \cite{Lin04}, so our coverage
is brief. Let $\cC$ be any category. We
define, as follows, a $\kay$-linear
category $\kay \cC$ called the
{\bf $\kay$-linearization} of $\cC$.
The objects of $\kay \cC$ are the objects
of $\cC$. Given objects $X$ and $Y$, then
$\kay \cC(X, Y)$ is the $\kay$-module
freely generated by $\cC(X, Y)$. The
composition for $\kay \cC$ is obtained
from the composition for $\cC$ by
$\kay$-linear extension. We have
$\id_X^{\kay \cC} = \id_X^\cC$. When
$\cC$ is small, the $\kay$-linearization
$\kay \cC$ is small, and we can pass to
the algebra $\kay \cC$, which we call the
{\bf category algebra} of $\cC$ over
$\kay$. As an equivalent definition, for
small $\cC$, the algebra $\kay \cC$ is
the algebra over $\kay$ such that
$\kay \cC$ has $\kay$-basis $\mor(\cC)$
and the multiplication operation on
$\kay \cC$ is the $\kay$-linear extension
of the composition operation, again with
the product of two incompatible
morphisms taken to be zero.

For any category $\cC$, a {\bf cocycle}
for  $\cC$ over $\kay$ is defined to be a
formal family of functions
$$\gamma_{X, Y, Z} \: : \: \kay^\times
  \leftarrow \cC(X, Y) \times \cC(Y, Z)$$
indexed by triples of objects $X$, $Y$, $Z$
of $\cC$, such that, dropping the subscripts,
$$\gamma(\theta, \phi)
  \gamma(\theta \comp \phi, \psi) =
  \gamma(\theta, \phi \comp \psi)
  \gamma(\phi, \psi)$$
for all $\cC$-morphisms $\theta$,
$\phi$, $\psi$ with $\theta \comp
\phi \comp \psi$ defined. We define the
{\bf twisted category} associated with
$\gamma$ to be the $\kay$-linear
category $\kay_\gamma \cC$ such that
$\kay_\gamma \cC = \kay \cC$ as
$\kay$-modules and the composition
${}_{\circ \gamma}$ satisfies
$$\phi \, {}_{\circ \gamma} \, \psi
  = \gamma(\phi, \psi) \,
  \phi \comp \psi$$
for all $\cC$-morphisms $\phi$ and $\psi$
such that $\phi \comp \psi$ is defined.
The associativity of the composition is clear.
It is easy to check that the identity
$\kay_\gamma \cC$- morphism on $X$ is
$\gamma(\id_X, \id_X)^{-1} \id_X^\cC$.
When $\cC$ is small, the algebra
$\kay_\gamma \cC$ is called the
{\bf twisted category algebra} associated
with $\gamma$.

We now turn to the subgroup category
$\cS$. For any group $R$, we write
$\cS(R)$ to denote the set of subgroups
of $R$. Goursat's Theorem, well-known
and easy to prove, provides a classification
of the subgroups of a direct product of
two groups.

\begin{pro}
\label{2.4} {\rm (Goursat's Theorem.)}
Let $R$ and $S$ be groups. Then
the condition
$$U = \{ x {}_\bullet U \times y U_\bullet
  : x \times y \in
  {}^\bullet U \times U^\bullet,
  x {}_\bullet U = \theta_U(y U_\bullet) \}$$
characterizes a bijective correspondence
$U \leftrightarrow ({}^\bullet U, {}_\bullet U,
\theta_U, U_\bullet, U^\bullet)$ between
the subgroups $U$ of $R {\times} S$ and
the quintuples $({}^\bullet U, {}_\bullet U,
\theta_U, U_\bullet, U^\bullet)$ such that
$R \geq {}^\bullet U \unrhd {}_\bullet U$
and $U_\bullet \unlhd U^\bullet \leq S$
and $\theta_U$ is an isomorphism
${}^\bullet U / {}_\bullet U \leftarrow
U^\bullet / U_\bullet$.
\end{pro}

Of course, the five parameters
${}^\bullet U$, ${}_\bullet U$, $\theta_U$,
$U_\bullet$, $U^\bullet$ depend on $R$
and $S$ as well as $U$. When we
apply the proposition to elements $U$ of
the set $\cS(R, S) = \cS(R {\times} S)$, we
shall usually be understanding the
codomain $R$ and the domain $S$ to
be implicit in the specification of $U$.
To guarantee disjointness of morphism
sets, a prerequisite condition in the
above definition of a category algebra,
we can understand the elements of
$\cS(R, S)$ to be triples having the form
$(R, U, S)$ where $U \leq R {\times} S$.
But let us not include that in our
notation. In the scenario of the
proposition, we write
$$U = \Delta({}^\bullet U, {}_\bullet U,
\theta_U, U_\bullet, U^\bullet) \; .$$
We abridge the notation in some
special cases, as follows. Let $A \leq R$
and $B \leq S$. Given an
isomorphism $\theta : A \leftarrow B$,
we write $\Delta(\theta) = \Delta(A, 1,
\theta, 1, B)$, which makes sense upon
identifying $A$ with $A/1$ and similarly
for $B$. Given $R \geq C \leq S$, we
write $\Delta(C) = \Delta(\id_C) =
\{ (c, c) : c \in C \}$. In particular, the
identity $\cS$-morphism on $R$ is
$\id_R^\cS = \Delta(R)$.

The following description of the star
product, though familiar to experts on
bisets, is worth briefly reviewing in the
notation that we shall be using. For
another account of the star product,
presented in the context of bisets, see
Bouc [Bou10, Chapter 2]. Let $R$, $S$,
$T$ be groups. Let $U \in \cS(R, S)$ and
$V = \cS(S, T)$. Write $W = U * V$. A
straightforward application of Zassenhaus'
Butterfly Lemma yields isomorphisms
$$\frac{{}^\bullet W}{{}_\bullet W} \cong
  \frac{(U {}^\bullet \cap {}^\bullet V)
  U {}_\bullet}{(U {}^\bullet \cap
  {}_\bullet V) U {}_\bullet} \cong
  \frac{U {}^\bullet \cap {}^\bullet V}
  {(U {}^\bullet \cap {}_\bullet V)
  (U {}_\bullet \cap {}^\bullet V)} \cong
  \frac{(U {}^\bullet \cap
  {}^\bullet V) {}_\bullet V}{(U {}_\bullet
  \cap {}^\bullet V) {}_\bullet V} \cong
  \frac{W {}^\bullet}{W {}_\bullet} \; .$$
The isomorphisms are the canonical
isomorphisms expressed in the
following variant of the diagram that
gives the Butterfly Lemma its name.
The four horizontal lines on the left
indicate how, via $\phi$, four subgroups
of ${}^\bullet U$ containing ${}_\bullet U$
correspond to four subgroups of
$U^\bullet$ containing $U_\bullet$. A
similar comment applies to the four
horizontal lines on the right, with
$\psi$ in place of $\phi$.

\smallskip
\noin \hspace{0.35in}
\begin{picture}(420,142)

\put(13,5){${}_\bullet U$}
\put(12,45){${}_\bullet W$}
\put(12,85){${}^\bullet W$}
\put(13,125){${}^\bullet U$}

\put(87,5){$U_\bullet$}
\put(64,45){$(U^\bullet \cap
  {}_\bullet V) U_\bullet$}
\put(64,85){$(U^\bullet \cap
  {}^\bullet V) U_\bullet$}
\put(85,125){$U^\bullet$}

\put(153,25){$(U^\bullet \cap {}_\bullet V)
  (U_\bullet \cap {}^\bullet V)$}
\put(181,65){$U^\bullet \cap {}^\bullet V$}

\put(297,5){${}_\bullet V$}
\put(279,45){$(U_\bullet \cap
  {}^\bullet V) {}_\bullet V$}
\put(279,85){$(U^\bullet \cap
  {}^\bullet V){}_\bullet V$}
\put(297,125){${}^\bullet V$}

\put(372,5){$V_\bullet$}
\put(370,45){$W_\bullet$}
\put(370,85){$W^\bullet$}
\put(372,125){$V^\bullet$}


\put(20,19){\line(0,1){20}}
\put(20,59){\line(0,1){20}}
\put(20,99){\line(0,1){20}}

\put(90,19){\line(0,1){20}}
\put(90,59){\line(0,1){20}}
\put(90,99){\line(0,1){20}}

\put(200,39){\line(0,1){20}}

\put(305,19){\line(0,1){20}}
\put(305,59){\line(0,1){20}}
\put(305,99){\line(0,1){20}}

\put(375,19){\line(0,1){20}}
\put(375,59){\line(0,1){20}}
\put(375,99){\line(0,1){20}}

\put(77,9){\line(-1,0){44}}
\put(56,49){\line(-1,0){21}}
\put(56,89){\line(-1,0){21}}
\put(77,129){\line(-1,0){44}}

\put(150,33){\line(-2,1){20}}
\put(173,73){\line(-4,1){42}}
\put(251,33){\line(2,1){20}}
\put(229,73){\line(4,1){42}}

\put(364,9){\line(-1,0){46}}
\put(362,49){\line(-1,0){17}}
\put(362,89){\line(-1,0){17}}
\put(364,129){\line(-1,0){46}}

\end{picture}

\smallskip
The next proposition repeats the above
description of $W$ in a more explicit way.

\begin{pro}
\label{2.5}
With the notation above, write
$M = U^\bullet\cap {}^\bullet V$ and
$N = (U^\bullet \cap {}_\bullet V)
(U_\bullet \cap {}^\bullet V)$.
Let $\oo{\phi} : {}^\bullet W / {}_\bullet W
\leftarrow M/N$ and $\oo{\psi} : M/N
\leftarrow W^\bullet / W_\bullet$ be the
isomorphisms induced by $\phi$ and
$\psi$, respectively. Then $\theta_W
= \oo{\phi} \comp \oo{\psi}$.
\end{pro}

For the sake of conciseness later, let
us make a pedantic distinction. We
understand a {\bf subquotient} of $R$ to
be a group having the form $M/N$,
constructed in the standard way, where
$R \geq M \unrhd N$. Note that $M/N$
determines the pair $(M, N)$, indeed,
$M$ is the unionset of $M/N$ while
$N$ is the identity element. We write
$[R]$ for the isomorphism class of $R$.
We call $S$ a {\bf factor group} of $R$
provided $S$ is isomorphic to a
subquotient of $R$. In that case, we
write $[S] \leq [R]$. Thus, we impose
a formal partial ordering $\leq$ on
the isomorphism classes of groups.

For any $U \in \cS(R, S)$, we define the
{\bf thorax} of $U$ to be the group
$\Theta(U)$, well-defined up to
isomorphism, such that
$${}^\bullet U / {}_\bullet U
  \cong \Theta(U) \cong
  U {}^\bullet / U {}_\bullet \; .$$
The latest proposition immediately yields
the following three corollaries.

\begin{cor}
\label{2.6}
With the notation above, $[\Theta(U)]
\geq [\Theta(W)] \leq [\Theta(V)]$.
\end{cor}

\begin{cor}
\label{2.7}
The morphism $U \in \cS(R, S)$ factorizes
through $\Theta(U)$. That is to say, there
exist $X \in \cS(R, \Theta(U))$ and
$Y \in \cS(\Theta(U), S)$ such that
$U = X * Y$. Furthermore, $U$ is an
$\cS$-isomorphism if and only if
$\Theta(U) \cong U$.
\end{cor}

\begin{cor}
\label{2.8}
We have a group isomorphism
$\mu_R : \Aut_\cS(R) \ni \Delta(\theta)
\mapsfrom \theta \in \Aut(R)$.
\end{cor}

We now introduce the cocycles that will
appear in the definition of $\Lambda$.
Let $F$, $G$, $H$, $I$ be any finite groups.
For brevity, we write $\ell(F) = \ell(|F|)$.
Let $U \in \cS(F, G)$, $V \in \cS(G, H)$,
$W \in \cS(H, I)$. We define
$$\sigma(U, V) =
  \ell(U_\bullet \cap {}_\bullet V) \; .$$
Consider the subgroup $A = \{ g {\times}
h : 1 {\times} g \in U, g {\times} h \in V,
h {\times} 1 \in W \} \leq G {\times} H$.
We have
$$\frac{U_\bullet \cap {}_\bullet
  (V * W)}{U_\bullet \cap {}_\bullet V}
  \cong \frac{{}^\bullet A}{{}_\bullet A}
  \cong \frac{A^\bullet}{A_\bullet} \cong
  \frac{(U * V)_\bullet \cap {}_\bullet W}
  {V_\bullet \cap {}_\bullet W} \; .$$
Thus, we have given a quick proof of
the equality, due to Boltje--Danz
\cite[3.5]{BD13},
$$|U_\bullet \cap {}_\bullet V| .
  |(U * V)_\bullet \cap {}_\bullet W| =
  |U_\bullet \cap {}_\bullet (V * W)| .
  |V_\bullet \cap {}_\bullet W| \; .$$
Since $\ell$ is a monoid homomorphism,
the conclusion can be expressed as
follows.

\begin{pro}
\label{2.9} {\rm (Boltje--Danz.)}
With the notation above,
$$\sigma(U, V * W) \sigma(V, W) =
  \sigma(U, V) \sigma(U * V, W) \; .$$
In other words, $\sigma$ is a cocycle
for the full subcategory of $\cS$ on
the class of finite groups.
\end{pro}

We write $\kay_\sigma \cS$ to denote
the twisted category associated with
$\sigma$. To avoid misunderstanding,
let us emphasize that, although the
objects of $\cS$ are arbitrary groups,
the cocycle $\sigma$ is defined only
for finite groups and the objects of
$\kay_\sigma \cS$ are arbitrary finite
groups. Retaining the notation above,
we write $s_U^{F, G}$ to denote $U$
as an element of $\kay_\sigma
\cS(F, G)$. The composition for
$\kay_\sigma \cS(F, G)$ is given by
$$s_U^{F, G} \comp s_V^{G, H} =
  \sigma(U, V) \, s_{U * V}^{F, H} \; .$$
Below, we shall usually omit the
$\comp$ symbol. The identity
$\kay_\sigma \cS$-morphism on
$G$ is $\id_G = s_{\Delta(G)}^{G, G}$.

\begin{lem}
\label{2.10}
Let $E$, $F$, $G$ be finite groups
and $U \in \cS(F, G)$ such that
$E \cong \Theta(U)$. Then there exist
$X \in \cS(F, E)$ and $Y \in \cS(E, G)$
satisfying $U = X * Y$. For any such
$X$ and $Y$, we have $s_U^{F, G}
= s_X^{F, E} s_Y^{E, G}$.
\end{lem}

\begin{proof}
Corollary \ref{2.7} implies the existence
of $X$ and $Y$. Corollary \ref{2.6}
implies that $\Theta(X) \cong E$. Hence,
$X_\bullet = 1$ and $\sigma(X, Y) = 1$.
\end{proof}

\section{Classification of simple modules}

To apply ring theoretic techniques, we
replace $\kay_\sigma \cS$ with the
small full subcategory
$$\Lambda_\cK =
  \kay_\sigma \cS_\cK \; .$$
Although $\Lambda_\cK$ is determined
by the triple $(\cK, \kay, \ell)$, we omit
$\kay$ and $\ell$ from the notation
because we shall always be treating them
as fixed. Let us point out that the
dependence on $\ell$ is a dependence
only on the values $\ell(q)$ where $q$
runs over the prime divisors of the
orders of the elements of $\cK$.
When no ambiguity can arise, we write
$\Lambda = \Lambda_\cK$. Employing
an abuse of notation discussed in the
previous section, the algebra associated
with the category $\Lambda$ will also
be written as $\Lambda$. We shall be
making a study of the algebra
$\Lambda$, which we view as coming
equipped with the complete family of
mutually orthogonal idempotents
$(\id_G^\Lambda : G \in \cK)$. We can
also view $\Lambda$ as the twisted
category algebra associated with the
restriction of $\sigma$ to $\cK$.
In this section, we shall be making
much use of the $\kay$-basis
$\{ s_U^{F, G} : F \in \cK \ni G,
U \in \cS(F, G) \}$ for $\Lambda$. We
call that basis the {\bf square basis}.

Let $F, G, H \in \cK$. Given group
isomorphisms $\phi : F \leftarrow G$
and $\psi : G \leftarrow H$, then
$$s_{\Delta(\phi)}^{F, G} \,
  s_{\Delta(\psi)}^{G, H} =
  s_{\Delta(\phi \psi)}^{F, H} \; .$$
So, specializing and reinterpreting the
map in Corollary \ref{2.8}, we have a
group monomorphism
$$\mu_G \: : \: \Aut_\Lambda(G) \ni
  s_{\Delta(\theta)}^{G, G}
  \mapsfrom \theta \in \Aut(G) \; .$$
Extending $\kay$-linearly, we obtain a
unity-preserving algebra monomorphism
$$\mu_G \: : \: \End_\Lambda(G)
  \leftarrow \kay \Aut(G) \; .$$
By Corollary \ref{2.7}, the set of
non-isomorphisms in the monoid
$\End_\cS(G) = \cS(G, G)$ is
$$\cS(G, G)_< = \{ U \in \cS(G, G)
  : [\Theta(U)] < [G] \} \; .$$
Let $\End_\Lambda(G)_<$ denote the
$\kay$-submodule of
$\End_\Lambda(G)$ with $\kay$-basis
$\{ s_U^{G, G} : U \in \cS(G, G)_< \}$. By
Corollary \ref{2.6}, $\End_\Lambda(G)_<$
is an ideal of $\End_\Lambda(G)$. The
next result follows.

\begin{pro}
\label{3.1}
Let $G \in \cK$. Then
$$\End_\Lambda(G) = \mu_G(\kay
  \Aut(G)) \oplus \End_\Lambda(G)_<$$
as the direct sum of a common-unity
subalgebra and an ideal.
\end{pro}

Digressing to introduce some general
notation, let $A$ and $B$ be rings and
$\theta : A \leftarrow B$ a
homomorphism. We write
${}_A \Ind_B^\theta$ for induction
to $A$-modules from $B$-modules
via $\theta$. We write
${}_B \Res_A^\theta$ for restriction
via $\theta$. When $\theta$ is an
isomorphism, we write
${}_A \Iso_B^\theta = {}_A \Ind_B^\theta
= {}_A \Res_B^{\theta^{-1}}$. We
sometimes omit the subscripts $A$
and $B$.

We define an {\bf $\cS$-seed} for
$\cK$ to be a pair $(E, W)$ such that
$E$ is a factor group of an element of
$\cK$ and $W$ is a simple
$\kay \Aut(E)$-module. Two such
pairs $(E, W)$ and $(E', W')$ are said
to be {\bf equivalent} provided there
exists a group isomorphism $\theta :
E \leftarrow E'$ such that $W \cong
\Iso^\theta(V)$. In Theorems \ref{3.4}
and \ref{3.5}, for arbitrary $\cK$, we
shall be describing a bijective
correspondence between the
isomorphism classes of simple
$\Lambda$-modules and the
equivalence classes of $\cS$-seeds
for $\cK$. The next theorem, too,
describes such a correspondence, but
under a strong hypothesis on $\cK$.
For other group categories, analogous
theorems, with similar hypotheses on
the set of objects, can be found in, for
instance, Th\'{e}venaz--Webb
\cite[Section 2]{TW95}, Bouc
\cite[4.3.10]{Bou10}. We say that $\cK$
is {\bf closed under factor groups up
to isomorphism} provided $\cK$ owns
an isomorphic copy of every subquotient
of every element of $\cK$. Note, this is
a strong and sometimes inconvenient
condition which excludes the important
case where $\cK$ consists of a single
non-trivial finite group.

\begin{thm}
\label{3.2}
Suppose $\cK$ is closed under factor
groups up to isomorphism. Then every
$\cS$-seed for $\cK$ is equivalent to an
$\cS$-seed $(E, W)$ for $\cK$ such that
$E \in \cK$. The isomorphism classes
$[S]$ of simple $\Lambda$-modules
$S$ and the equivalence classes $[E, W]$
of $\cS$-seeds $(E, W)$ for $\cK$
are in a bijective correspondence
$[S] \leftrightarrow [E, W]$ characterized
by the condition that, if $E \in \cK$, then,
with respect to the partial ordering of
isomorphism classes of groups, $[E]$ is
minimal subject to $S(E) \neq 0$ and
$W \cong \Res^{\mu_E}(S(E))$.
\end{thm}

\begin{proof}
The first sentence of the conclusion is
obvious. By Lemma \ref{2.10} and the
hypothesis on $\cK$, every
$\Lambda$-morphism $F \leftarrow G$
factorizes through a group $T \in \cK$
such that $[F] \geq [T] \leq [G]$.
Perforce, in the terminology of [BD16],
the $\kay$-linear category $\Lambda$
is admissible with respect to the above
partial ordering. The required conclusion
now follows from [BD16, 2.4].
\end{proof}

To generalize the correspondence in
the latest theorem, we need a lemma.

\begin{lem}
\label{3.3}
Suppose $\cK$ is closed under
subquotients up to isomorphism.
Let $[S] \leftrightarrow [E, W]$ in the
notation of the latest theorem. Let
$G \in \cK$. Then $S(G) \neq 0$ if and
only if $E$ is a factor group of $G$.
\end{lem}

\begin{proof}
In one direction, the conclusion is
already part of the characterization
of $S$. Conversely, suppose $G$ has
a subquotient $B/Y$ isomorphic to
$E$. Replacing $E$ with an isomorphic
copy, we may assume that $E \in \cK$.
Let $x \in S(E) - \{ 0 \}$. Let
$\phi : E/1 \leftarrow B/Y$ be an
isomorphism. Let
$I = \Delta(E, 1, \phi, Y, B)$ and
$J = \Delta(B, Y, \phi^{-1}, 1, E)$.
Then $I * J = \Delta(E)$ and
$\sigma(I, J) = \ell(Y)$, so
$$s_I^{E, G} s_J^{G, E} x
  = \ell(Y) s_{\Delta(E)}^{E, E} x
  = \ell(Y) x \neq 0 \; .$$
So the element $s_J^{G, E} x \in
S(G)$ is non-zero.
\end{proof}

Now let $\cK$ be arbitrary. Given
an $\cS$-seed $(E, W)$ for $\cK$,
we define $S_{E, W}$ to be the
simple $\Lambda$-module,
well-defined up to isomorphism,
defined as follows. Choose any
set $\cK'$ of finite groups such
that $\cK \subseteq \cK'$ and
$\cK'$ is closed under subquotients
up to isomorphism. Let $\Lambda'
= \Lambda_{\cK'}$. Since
$(E, W)$ is an $\cS$-seed for $\cK'$,
there exists an isomorphically unique
simple $\Lambda'$-module
$S'$ corresponding to $(E, W)$ as
in Theorem \ref{3.2}. We define
$S_{E, W} = \Lambda S'$, which
is a simple $\Lambda$-module by
Proposition \ref{2.1} and
Lemma \ref{3.3}. To check that
$S_{E, W}$ is independent of the
choice of $\cK'$, let $\cK''$ be
another such set of finite groups.
Write $\cK''' = \cK' \cup \cK''$.
Let $S''$ and $S'''$ be the simple
modules defined in the same way
as $S'$ but with $\cK''$ and
$\cK'''$, respectively, in place of
$\cK'$. By considering a diamond
diagram of corner embeddings, we
deduce that $\Lambda S' \cong
\Lambda S''' \cong \Lambda S''$. The
well-definedness is now established.
We say that $S_{E, W}$ has
{\bf $\cS$-seed} $(E, W)$ and
{\bf minimal group} $E$. We say
that $(E, W)$ and $E$ are
{\bf associated} with $S_{E, W}$.

Let us point out that the
$\Lambda$-modules $\Lambda S'$,
$\Lambda S''$, $\Lambda S'''$ can
be viewed as restrictions. Indeed, in
the scenario of Proposition \ref{2.1},
we can make an identification
$BN = {}_B \Res {}_A (N)$ for any
$A$-module $N$. The next remark
consolidates an observation that is
already implicit above.

\begin{rem}
\label{3.4}
Let $(E, W)$ be an $\cS$-seed for $\cK$.
Let $\cL \subseteq \cK$ such that $E$
is a factor group of a group in $\cL$.
Then $(E, W)$ is an $\cS$-seed for $\cL$
and the restricted $\Lambda_\cL$-module
$\Lambda_\cL . S_{E, W}$ has $\cS$-seed
$(E, W)$.
\end{rem}

Let us emphasize that, in the next result,
$\cK$ is arbitrary. In particular, we do
not require $\cK$ to own isomorphic
copies the minimal groups associated
with the simple
$\Lambda$-modules. This is in contrast
to the situation for some other kinds of
group functors, where work without
such an ownership condition tends to
be very difficult. Several examples in
Bouc--Stancu--Th\'{e}venaz
\cite[Section 13]{BST13} show that, for
biset functors, there are no direct
analogues of the latest lemma or the
two theorems below in this section. See
also the discussion of the ``vanishing
problem'' in Rognerud \cite{Rog19}.

\begin{thm}
\label{3.5}
The condition $S \cong S_{E, W}$
characterizes a bijective
correspondence $[S] \leftrightarrow
[E, W]$ between the isomorphism
classes $[S]$ of simple
$\Lambda$-modules $S$ and the
equivalence classes $[E, W]$ of
$\cS$-seeds $(E, W)$ for $\cK$.
\end{thm}

\begin{proof}
Let $\cK'$ and $\Lambda'$ be as
defined just above. Theorem
\ref{3.2} describes the simple
$\Lambda'$-modules. The argument
is completed by applying Proposition
\ref{2.1} to the corner embedding
$\Lambda' \hookleftarrow \Lambda$
and then making use of Lemma
\ref{3.3} and Remark \ref{3.4}.
\end{proof}

The next theorem describes the same
correspondence in a more intrinsic
way, without mentioning any
extension of the set of objects. Let
$E$ be a finite group, let $\cK \ni G
\geq B \unrhd Y$ and let $\phi : E
\leftarrow B/Y$ be an isomorphism.
We write
$$\mu_\phi \: : \: \End_\Lambda(G)
  \leftarrow \kay \Aut(E)$$
for the algebra monomorphism such
that, given $\epsilon \in \Aut(E)$, then
$$\mu_\phi(\epsilon) = \ell(Y)^{-1}
  s_{\Delta(B, Y, \phi^{-1} \epsilon
  \phi, Y, B)}^{G, G} \; .$$

\begin{thm}
\label{3.6}
The correspondence in the previous
theorem is characterized as follows.
Given a simple $\Lambda$-module $S$
then, up to equivalence, the $\cS$-seed
$(E, W)$ for $\cK$ associated with
$S$ is determined by the following three
conditions:

\noin {\bf (a)} For any $G \in \cK$, we
have $S(G) \neq 0$ if and only if
$[E] \leq [G]$.

\noin {\bf (b)} Given $\cK \ni G
\geq B \unrhd Y$, an isomorphism
$\phi : E \leftarrow B/Y$ and an
idempotent $k$ of $\kay \Aut(E)$,
then $\mu_\phi(k) S(G) \neq 0$ if
and only if $kW \neq 0$.

\noin {\bf (c)} The isomorphism class
$[E]$ is minimal in the sense that, given
an $\cS$-seed $(E', W')$ for $\cK$
satisfying conditions (a) and (b),
then $[E] \leq [E']$.
\end{thm}

\begin{proof}
By Proposition \ref{2.1} and Remark
\ref{3.4}, we may assume that $S$ has
$\cS$-seed $(E, W)$ and that $E \in \cK$.
It suffices to deduce conditions (a), (b), (c).
Lemma \ref{3.3} yields condition (a) which,
together with Theorem \ref{3.2}, implies
condition (c). Let
$I = \Delta(E, 1, \phi, Y, B)$ and
$J = \Delta(B, Y, \phi^{-1}, E, 1)$.
Given $\epsilon \in \Aut(E)$, then
$$s_I^{E, G} \mu_\phi(\epsilon) s_J^{G, E}
  = \ell(Y) \mu_E(\epsilon) \; .$$
By $\kay$-linearity, the same equality
holds with $k$ in place of $\epsilon$.
We have $kW \neq 0$ if and only if
$\mu_E(k) S(E) \neq 0$. In that case,
$\mu_\phi(\epsilon) s_J^{G, E} S(E)
\neq 0$, hence $\mu_\phi(\epsilon)
S(G) \neq 0$. Conversely, suppose
$S(G) \neq 0$. Noting that
$$s_J^{G, E} \mu_E(\epsilon) s_I^{E, G}
  = \ell(Y) \mu_\phi(\epsilon) \; .$$
and arguing as before, we obtain
$kW \neq 0$. We have deduced
condition (b).
\end{proof}

Let us mention that, if we were to
generalize by allowing $\ell$ to be
any homomorphism of multiplicative
monoids $\kay^\times \leftarrow \NN
- \{ 0 \}$ then, adopting a suitable
extension of the notion of a twisted
category algebra, the discussion in
this section would carry through,
up to and including Theorem
\ref{3.2}, but the proof of
Lemma \ref{3.3} would no longer
be valid. We do not know of any
examples for which the conclusion
of that lemma would fail.

Recall, for a ring $A$, a primitive
idempotent of the centre $Z(A)$ is
called a {\bf block} of $A$. When
$A$ is unital and semisimple, the
blocks of $A$ are in a bijective
correspondence with the simple
$A$-modules up to isomorphism,
each block corresponding to the
isomorphically unique simple module
of the associated block algebra.
If $\cK$ is finite then, for each
$\cS$-seed $(E, W)$ for $\cK$, we
let $b_{E, W}$ denote the block of
$\Lambda$ that acts as the identity
on $S_{E, W}$.

\begin{example}
\label{3.7} Let $q$ be a prime and
$\cK = \{ G \}$ with $G \cong C_q$, the
cyclic group with order $q$. Suppose
$\kay$ has a root of unity with order
$q - 1$. Then the algebra $\Lambda =
\End_\Lambda(G)$ has a $\kay$-basis
consisting of the elements
$$s_0 = s_{1 \times 1}^{G, G} \; , \;\;\;\;
  s_{01} = s_{1 \times G}^{G, G} \; , \;\;\;\;
  s_{10} = s_{G \times 1}^{G, G} \; , \;\;\;\;
  s_{11} = s_{G \times G}^{G, G} \; , \;\;\;\;
  s_d = s_{\Delta(d)}^{G, G}$$
where $\Delta(d) = \{ (g^d, g) : g \in G \}$
and $d$ runs over the elements of the
unit group $(\ZZ / q)^\times$ of the ring
$\ZZ / q$ of modulo $q$ congruence
classes. The $\cS$-seeds for $\cK$ are
$(1, 1)$, $(G, \zeta)$, $(G, \chi)$ where
$\zeta$ is a trivial $\kay \Aut(G)$-module
and $\chi$ runs over representatives of the
$q - 2$ isomorphism classes of non-trivial
simple $\kay \Aut(G)$-modules, vacuously
when $q = 2$. Write $\lambda = \ell(q)$
and $r = \lambda s_0 - s_{01} - s_{10}
+ s_{11}$.

\noin {\bf (1)} If $\lambda = 1$, then
$\kay r$ is a nilpotent ideal of
$\Lambda$. In particular, $\Lambda$ is
not semisimple.

\noin {\bf (2)} Suppose $\lambda
\neq 1$. Then $\Lambda$ is semisimple.
Identifying each $\chi$ with the
associated irreducible character, the
blocks of $\Lambda$ are
$$b_{1,1} = \frac{r}{\lambda-1} \; ,
  \;\;\;\;\;\; b_{G, \zeta} =
  \frac{-r}{\lambda - 1} + \frac{1}
  {q - 1} \sum_d s_d \; , \;\;\;\;\;\;
  b_{G, \chi} = \frac{1}{q - 1}
  \sum_d \chi(d^{-1}) s_d \; .$$
We have $\Lambda b_{1, 1} \cong
\Mat_2(\kay)$ and $\Lambda
b_{G, \zeta} \cong \Lambda
b_{G, \chi} \cong \kay$.
\end{example}

\begin{proof}
This exercise in laborious but
straightforward calculation can be
done in many different ways. Let us
sketch a fairly quick route. Suppose
$\lambda \neq 1$. It is easy to check
that the regular $\Lambda$-module
has a submodule $S$ with
$\kay$-basis $\{ s_0, s_{10} \}$ and
representation given, with respect to
that basis, by
$$\openmat 1 & 1 \\ 0 & 0 \closemat
  \mapsfrom s_0 \, , \;\;
  \openmat 1 & \lambda \\ 0 & 0 \closemat
  \mapsfrom s_{01} \, , \;\;
  \openmat 0 & 0 \\ 1 & 1 \closemat
  \mapsfrom s_{10} \, , \;\;
  \openmat 0 & 0 \\ 1 & \lambda \closemat
  \mapsfrom s_{11} \, , \;\;
  \openmat 1 & 0 \\ 0 & 1 \closemat
  \mapsfrom s_d \, .$$
Plainly, $S$ is simple and
$\End_\Lambda(S) \cong \kay$. By
Theorem \ref{3.5}, $\Lambda$ has
exactly $q$ simple modules up to
isomorphism. By counting dimensions,
$\Lambda$ is semisimple. Using the
above matrices, it is easy to show
that $Z(\Lambda)$ has a basis
consisting of $r$ and the elements
$s_d$. We have $b_{1, 1} \in
Z(\Lambda) \cap \Lambda_< =
\kay r$. The formula for $b_{1, 1}$
follows easily.

Let $\psi$ run over all $q - 1$ of the
irreducible characters of $\kay \Aut(G)$.
Define $b_\psi = \sum_d \psi(d^{-1})
s_d$. Direct calculations show that
$\sum_\psi b_\psi = 1$ as a sum of
mutually orthogonal idempotents of
$Z(\Lambda)$. By considering actions
on simple $\kay \Aut(G)$-modules and
noting that $b_{1, 1} b_\zeta = b_{1, 1}$,
we deduce that $b_{1, 1} + b_{G, \zeta} =
b_\zeta$ and each $b_{G, \chi} = b_\chi$.
\end{proof}

Generally, we can identify $\Lambda$ with
$\kay \cS_\cK$ as $\kay$-modules by
identifying each $s_U^{F, G}$ with the
element $U \in \cS(F, G)$. In fact, that
identification of underlying $\kay$-modules
is already implicit in our construction of
$\Lambda$. The example shows that the
isomorphism in Theorem \ref{1.2} can
depend on $\ell$. Indeed, in the
notation of the example, putting
$\lambda \neq 1$ then, as
$\kay$-submodules of $\kay \cS_\cK$,
the block algebras $\Lambda b_{1, 1}$
and $\Lambda b_{G, \chi}$ are
independent of $\ell$, but the block
algebra $\Lambda b_{G, \zeta}$ does
depend on $\ell$. Moreover, again as
$\kay$-submodules of $\kay \cS_\cK$,
the first summand on the right-hand side of
the algebra isomorphism
$$\Lambda \cong \Mat_2(\kay \Aut(1))
  \oplus \Mat_1(\kay \Aut(G))$$
is constant, but the other summand
varies.

\section{Implications of semisimplicity}

In this section, we prove
Theorem \ref{1.2}.

For convenience, we constrain the notion
of a $\Lambda$-module by deeming that,
for any $\Lambda$-module $M$, the
$\kay$-module $M(G)$ is
finite-dimensional for all $G \in \cK$,
equivalently, $eM$ is finite-dimensional
for all idempotents $e$ of $\Lambda$.
The condition ensures that every simple
factor of $M$ has a well-defined finite
multiplicity.
Given $G \in \cK$, then the module
$M(G)$ of $\End_\Lambda(G)$ has
submodule $\End_\Lambda(G)_< .
M(G)$ and we can define a
$\kay \Aut(G)$-module
$$\oo{M}(G) =
  \Res^{\mu_G} (M(G)
  / \End_\Lambda(G)_< . M(G)) \; .$$
Given an $\cS$-seed $(E, W)$ for
$\cK$, we let $m_{E, W}(M)$ denote the
multiplicity of $S_{E, W}$ as a simple
factor of $M$. When $\Lambda$ is
locally semisimple,
$$M \cong \bigoplus
  m_{E, W}(M) \, S_{E, W}$$
summed over representatives
$(E, W)$ of the equivalence classes
of $\cS$-seeds for $\cK$.

\begin{lem}
\label{4.1}
Suppose $\Lambda$ is locally
semisimple. Let $(E, W)$ be an
$\cS$-seed for $\cK$ such that
$E \in \cK$. Let $M$ be a
$\Lambda$-module. Then
$m_{E, W}(M)$ is the multiplicity
of $W$ in $\oo{M}(E)$.
\end{lem}

\begin{proof}
Write $\cE = \End_\Lambda(E)$. We
may assume that $M$ is simple.
Let $(E', W')$ be an $\cS$-seed for
$\cK$ associated with $M$. If
$\oo{M}(E) = 0$, then $[E', W'] \neq
[E, W]$ and $m_{E, W} = 0$. So we
may assume that $\oo{M}(E) \neq 0$.
By condition (c) of Theorem \ref{3.6},
$[E'] \leq [E]$. Let $E \geq B \unrhd Y$
and let $\phi : E' \leftarrow B/Y$ be an
isomorphism. By condition (b) of the
same theorem, $\mu_\phi(\epsilon')
M(E) \neq 0$ for some $\epsilon' \in
\Aut(E')$. For a contradiction, suppose
$[E'] < [E]$. The proof of the theorem
shows that $\mu_\phi(\epsilon')
\in \cE_<$. Therefore,
$\cE_< . M(E) \neq 0$. But, applying
Proposition \ref{2.1} to the corner
subalgebra $\cE$ of $\Lambda$,
we deduce that $M(E)$ is a simple
$\cE$-module. Therefore
$\cE_< . M(E) = M(E)$, contradicting
an assumption on $M$. We have
shown that $M$ has minimal group
$E$. That reduces to the case
$\cK = \{ E \}$, for which the
required conclusion is clear.
\end{proof}

Now letting $(E, W)$ be any
$\cS$-seed for $\cK$, we let $m_W$
denote the multiplicity of $W$ in the
regular $\kay \Aut(E)$-module. Of
course, when $\kay$ is algebraically
closed, $m_W = \dim_\kay(W)$. For
$G \in \cK$, we let $n_E^G$ denote
the number of subquotients of $G$
isomorphic to $E$.

\begin{lem}
\label{4.2}
Suppose $\Lambda$ is locally
semisimple. Given $G \in \cK$ and
an $\cS$-seed $(E, W)$ for $\cK$,
then $m_{E, W}(\Lambda .
\id_G^\Lambda) = n_E^G m_W$.
\end{lem}

\begin{proof}
Let $\cK' = \cK \cup \{ E \}$ and
$\Lambda' = \Lambda_{\cK'}$. Noting
that $\id_G^{\Lambda'} =
s_{\Delta(G)}^{G, G} = \id_G^\Lambda$,
Remark 3.4 allows us to replace $\cK$
with $\cK'$. So we may assume that
$E \in \cK$. Write $M = \Lambda .
\id_G^\Lambda$ and, again, $\cE =
\End_\Lambda(E)$. Then $M =
\bigoplus_{F \in \cK} \Lambda(F, G)$
and $M(E) = \Lambda(E, G)$. Given
$U \in \cS(E, G)$ and writing
$U = \Delta(A, X, \phi, Y, B)$, then
$$s_{\Delta(A, X, \id, X, A)}^{E, E}
  s_U^{E, G} = \ell(X) s_U^{E, G}$$
which belongs to $\cE_< . M(E)$
unless $A = E$ and $X = 1$. So
$$\oo{M}(E) =
  \bigoplus_{\phi, B/Y} \kay
  (s_{\Delta(E, 1, \phi, Y, B)}^{E, G}
  + \cE_< . M(E))$$
as $\kay$-modules, where $B/Y$ runs
over the subquotients of $G$ isomorphic
to $E$ and $\phi$ runs over the
isomorphisms $E \leftarrow B/Y$.
Given $\epsilon \in \Aut(E)$, then
$$\mu_E(\epsilon)
  s_{\Delta(E, 1, \phi, Y, B)}^{E, G}
  = s_{\Delta(E, 1, \epsilon \comp
  \phi, Y, B)}^{E, G} \; .$$
So $\oo{M}(E)$ is isomorphic
to the direct sum of $n_E^G$ copies
of the regular $\kay \Aut(E)$-module.
An appeal to Lemma \ref{4.1}
completes the argument.
\end{proof}

\begin{lem}
\label{4.3}
Given an $\cS$-seed $(E, W)$ for
$\cK$, then $\End_\Lambda(S_{E, W})
\cong \End_{\kay \Aut(E)}(W)$ as
an isomorphism of division algebras
over $\kay$.
\end{lem}

\begin{proof}
The ring isomorphisms in the proof
of Lemma \ref{2.2} are $\kay$-linear
when, in the notation of that lemma,
$A$ is an algebra over $\kay$ and
$B$ is a subalgebra.
\end{proof}

Given an $\cS$-seed $(E, W)$ for
$\cK$, we let $\Delta_W$ denote
the opposite ring of
$\End_{\kay \Aut(E)}(W)$. Up to
isomorphism, $\Delta_E$ is
determined by the condition that
$$\kay \Aut(E) b_W \cong
  \Mat_{m_W}(\Delta_W)$$
where $b_W$ is the block of
$\kay \Aut(E)$ fixing $W$.

\begin{thm}
\label{4.4}
Suppose $\cK$ is finite and $\Lambda$
is semisimple. Given an $\cS$-seed
$(E, W)$ for $\cK$, then
$$\Lambda b_{E, W} \cong
  \Mat_{n_E m_W}(\Delta_{E, W})$$
where $\Delta_{E, W}$ is the opposite
algebra of $\End_{\kay \Aut(E)}(W)$.
\end{thm}

\begin{proof}
The regular $\Lambda$-module
decomposes as ${}_\Lambda \Lambda
= \bigoplus_{G \in \cK} \Lambda .
\id_G^\Lambda$. Since $n_E = \sum_G
n_E^G$, Lemma \ref{4.2} gives
$m_{E, W}({}_\Lambda \Lambda)
= n_E m_W$. So
$${}_\Lambda \Lambda b_{E, W}
  \cong n_E m_W S_{E, W}$$
as $\Lambda$-modules. To complete
the argument, we apply Lemma
\ref{4.3}.
\end{proof}

Theorem \ref{1.2} follows because
$\Lambda = \bigoplus_{E, W}
\Lambda b_{E, W}$, summed over the
$\cS$-seeds $(E, W)$ up to equivalence.

\section{The round basis}

After Boltje--Danz \cite{BD13}, we
introduce a basis for $\Lambda$, called
the {\bf round basis}. We review
their characterization
\cite[Sections 2, 3]{BD13} of the product
of two elements of that basis. Our
reformulation, in terms of changes of
coodinates rather than changes of
associative operation, will suit our
applications in later sections.

We refer to Stanley \cite[Chapter 3]{Sta11}
for some terminology pertaining to posets.
Let $\cP$ be a poset such that every finite
subset of $\cP$ generates a finite order
ideal, equivalently, for all $u \in \cP$,
the principal order ideal $(\dash, u]_\cP
= \{ v \in \cP : v \leq u \}$ is finite. See
\cite[3.7, 3.8]{Sta11} for an introduction
to the theory of the M\"{o}bius function
$\mob_\cP : \ZZ \leftarrow \cP {\times}
\cP$. A well-known and straightforward
generalization of \cite[3.7.1]{Sta11}
asserts that, given an abelian group
$\cA$ and funtions $\sigma, \tau :
\cA \leftarrow \cP$, then the following
two conditions are equivalent:

\smallskip
\noin $\bullet$ we have $\sigma(u)
= \sum_{v \in (\dash, u]_\cP}
\tau(v)$ for all $u \in \cP$,

\smallskip
\noin $\bullet$ we have $ \tau(u) =
\sum_{v \in \cP} \sigma(v) \mob(v, u)$
for all $v \in \cP$.

\smallskip
\noin We call $\sigma$ the {\bf sum
function} of $\tau$. We call $\tau$ the
{\bf totient function} of $\sigma$. Let
us mention that the terminology reflects
the intended sense in the origin of the
word {\it totient}, Sylvester [Sy888].

A {\bf downward retraction} of $\cP$
is defined to be a decreasing idempotent
endomorphism of $\cP$, we mean, a
function $\rho : \cP \leftarrow \cP$ such
that $\rho^2(u) = \rho(u) \leq u$ for all
$u \in \cP$. All the retractions we consider
will be downward retractions. The next
result is in Boltje--Danz [BD13, 4.1].
Let us give a quick alternative proof.

\begin{lem}
\label{5.1} {\rm (Boltje--Danz.)}
Let $\cP$ and $\cA$ be as above, $\rho$
a retraction of $\cP$ and
$\sigma : \cA \leftarrow \cP$ a function
such that $\sigma(u) = \sigma(\rho(u))$
for all $u \in \cP$. Write $\sigma'$ for
the restriction of $\sigma$ to the
subposet $\cP' = \rho(\cP)$. Let $\tau$
and $\tau'$ be the totient functions
of $\sigma$ and $\sigma'$ on $\cP$
and $\cP'$, respectively. Then $\tau(v)
= \tau'(v)$ for all $v \in \cP'$ and
$\tau(v) = 0$ for all $v \in \cP - \cP'$.
\end{lem}

\begin{proof}
Let $v \in \cP$. First consider the
case $v \not\in \cP'$. Writing
$v' = \rho(v)$, then
$$0 = \sigma(v) - \sigma(v') = \tau(v)
  + {\sum}_w \tau(w)$$
where $w$ runs over those elements
of $\cP$ such that $v' \not\geq w < v$.
Since $\rho$ is an endomorphism, each
$\rho(w) \leq v' \neq w$, hence
$w \not\in \cP'$. By an inductive argument
on the height of $v$, we may assume that
each $\tau(w) = 0$, hence $\tau(v) = 0$.
Now consider the case $v \in \cP'$.  By
the conclusion for the other case,
$$\tau(v) + \!\! \sum_{w \in
  (\dash, v)_{\cP'}} \!\! \tau(w) =
  \sum_{w \in (\dash, v]_\cP} \tau(w)
  = \sigma(v) = \sigma'(v) =
  \tau'(v) + \!\! \sum_{w \in
  (\dash, v)_{\cP'}} \!\! \tau'(w)$$
where the notation indicates that
two of the summations are over open
intervals. Another inductive argument
on height yields $\tau(v) = \tau'(v)$.
\end{proof}

Let $F$, $G$, $H$ be finite groups. We
write $\mob(F, G)$ for the value, at
$(F, G)$, of the M\"{o}bius function of
the formal poset of finite groups,
partially ordered by inclusion. Let
$I \in \cS(F, G)$, $J \in \cS(G, H)$,
$K \in \cS(F, H)$. We define
$$\cP_K^{I, J} = \{ (U, V) \in \cS(F, G)
  \times \cS(G, H) : K \leq U * V, (U, V)
  \leq (I, J) \}$$
as a subposet of the direct product
$\cS(F, G) {\times} \cS(G, H)$. We define
$$\tau_K^{I, J} = \sum_{(U, V) \in
  \cP_K^{I, J}} \mob(U, I) \mob(V, J)
  \sigma(U, V) \; .$$
We call the pair $(I, J)$ {\bf strongly
compatible} provided $I^\bullet =
{}^\bullet J$. Note, the condition implies
that ${}^\bullet (I * J) = {}^\bullet I$
and $(I * J)^\bullet = J^\bullet$. Given
$W$ such that $K \leq W \in \cS(F, H)$,
we call $K$ an {\bf adequate subgroup}
of $W$ provided ${}^\bullet K =
{}^\bullet W$ and $K^\bullet =
W^\bullet$. Let $\ad(W)$ denote the
set of adequate subgroups of $W$. We
point out that, if $(I, J)$ is strongly
compatible and $K \in \ad(I * J)$,
then ${}^\bullet K = {}^\bullet I$ and
$K^\bullet = J^\bullet$.

We define $\{ t_I^{F, G} : I \in
\cS(F, G) \}$ to be the $\kay$-basis
for $\kay_\sigma \cS(F, G)$ given
by the two equivalent equalities
$$s_U^{F, G} = \sum_{I \in \cS(U)}
  t_I^{F, G} \; , \dozspace \;\;\;\;\;\;
  t_I^{F, G} = \sum_{U \in \cS(I)}
  \mob(U, I) \, s_U^{F, G}$$
where the first equality holds for all
$U \in \cS(F, G)$, the second, for all
$I \in \cS(F, G)$. To confirm the
equivalence, observe that the functions
$\kay_\sigma \cS(F, G) \leftarrow
\cS(F {\times} G)$ given
by  $s_U^{F, G} \mapsfrom U$ and
$t_I^{F, G} \mapsfrom I$ are, respectively,
the sum function and the totient function
of each other. We point out that the
basis $\{ t_I^{F, G} : I \}$ of
$\kay_\sigma \cS(F, G)$ can be
regarded simply as a basis of the
$\kay$-module $\kay \cS(F, G)$,
well-defined independently of $\ell$.
However, we shall be concerned with
composites of the basis elements
in $\kay_\sigma \cS$. The composites
do depend on $\ell$.

The set $\{ t_I^{F, G} : F \in \cK \ni G,
I \in \cS(F, G) \}$ is another $\kay$-basis
for $\Lambda$. We call it the {\bf round
basis}.  Let us mention that the round
basis has a long history, a version of it
going back to the 19th century, implicitly
appearing in Burnside's celebrated
table of marks. We shall be working with
the round basis of $\Lambda$ in later
sections. For now, though, there is no
need to pass down to a small
subcategory of $\kay_\sigma \cS$.

The next result is [BD13, 3.9, 4.2, 4.3].
In some of our applications of the
formulas for $\tau_K^{I, J}$, we shall
be considering the poset retraction
that appears in the proof.

\begin{thm}
\label{5.2} {\rm (Boltje--Danz.)}
Let $F$, $G$, $H$ be finite groups.
Let $I \in \cS(F, G)$ and
$J \in \cS(G, H)$. Then:

\noin {\bf (1)} We have
$$t_I^{F, G} t_J^{F, G} = \sum_{K \in
  \cS(F, H)} \tau_K^{I, J} t_K^{I, J} \; .$$

\noin {\bf (2)} For any $K \in \cS(F, H)$,
if $\tau_K^{I, J} \neq 0$, then $(I, J)$ is
strongly compatible and $K \in \ad(I * J)$.

\noin {\bf (3)} If $t_I^{F, G} t_J^{G, H}
\neq 0$, then $(I, J)$ is strongly
compatible.

\noin {\bf (4)} If $(I, J)$ is strongly
compatible then, for all $K \in \ad(I * J)$,
we have
$$\tau_K^{I, J} = \sum_{(U, V) \in
  \cR_K^{I, J}} \mob_{\cR_K^{I, J}}
  ((U, V), (I, J)) \sigma(U, V)$$
where $\cR_K^{I, J} = \{ (U, V) \in
\cP_K^{I, J} : U^\bullet = {}^\bullet V \}$.
\end{thm}

\begin{proof}
A straightforward manipulation yields
part (1). Note that, if $K \not\leq I * J$,
then $\cP_K^{I, J} = \emptyset$ and
$\tau_K^{I, J} = 0$. Now fix $K \leq I * J$.
By the convexity of $\cP_K^{I, J}$ in
$\cS(F, G) {\times} \cS(G, H)$ together
with the formula in \cite[3.8.2]{Sta11}
for the M\"{o}bius function of a direct
product of posets,
$$\mob(U, I) \mob(V, J) =
  \mob_{\cP_K^{I, J}} ((U, V), (I, J))$$
for each $(U, V) \in \cP_K^{I, J}$. Let
$$\Gamma_K(U, V) = \{ f {\times} g
  {\times} h \in F {\times} G {\times} H
  : f {\times} g \in U, f {\times} h \in K,
  g {\times} h \in V \} \; .$$
Let $S_{U, V} = \{ f {\times} g : f {\times}
g {\times} h \in \Gamma_K(U, V) \}$ and
$T_{U, V} = \{ g {\times} h : f {\times} g
{\times} h \in \Gamma_K(U, V) \}$,
which are subgroups of $U$ and $V$,
respectively. The function
$$\rho_K^{I, J} : \cP_K^{I, J} \ni
  (S_{U, V}, T_{U, V}) \mapsfrom (U, V)
  \in \cP_K^{I, J}$$
is easily shown to be a retraction that
preserves $\sigma$ and has image
$\cR_K^{I, J}$. Parts (2), (3), (4) now
follow from Lemma \ref{5.1}.
\end{proof}

\section{Semisimplicity and endomorphism
algebras}

Rognerud [Rog19, 7.3] gave a quick proof
of Bouc's special case of Theorem
\ref{1.4} by showing that the given category
is locally semisimple if and only if the
endomorphism algebras of all the
objects are semisimple. In this section,
we shall pursue a similar theme, though
the results are not directly analogous.

The following observations, employed
only incidentally in the next proof, will
be used more substantially in the next
section. Let $R$, $S$, $T$ be groups
and $U \in \cS(R, S)$ and $V \in \cS(S, T)$.
The {\bf opposite} of $U$ is defined to
be the morphism $U^\circ \in \cS(S, R)$
such that $U^\circ = \{ s {\times} r :
r {\times} s \in U \}$. Plainly,
$(U * V)^\circ = V^\circ * U^\circ$.
Now let $F$, $G$, $H$ be finite groups.
Let $U \in \cS(F, G)$ and $V \in \cS(G, H)$.
Then $\sigma(U, V) = \sigma(V^\circ,
U^\circ)$. So, now taking $F$, $G$, $H$
to be elements of $\cK$, there is a
self-inverse antiautomorphism
$\Lambda \ni u^\circ \leftrightarrow
u \in \Lambda$ given by
$(s^{F, G}_U)^\circ = s_{U^\circ}^{G, F}$.

\begin{lem}
\label{6.1}
Let $F$, $G$, $H$ be finite groups. Let
$I \in \cS(F, G)$ and $J \in \cS(G, H)$.

\noin {\bf (1)} Suppose $I$ and $J$ have
the form $I = \Delta(A, 1, \phi, 1, B)$ and
$J = \Delta(B, Y, \psi, Z, C)$. Then
$t_I^{F, G} t_J^{G, H} = t_K^{F, H}$ where
$K = \Delta(A, \phi(Y), \uu{\phi} \comp
\psi, Z, C)$ and $\uu{\phi} : A / \phi(Y)
\leftarrow B/Y$ is the isomorphism
induced by $\phi$.

\noin {\bf (2)} Suppose $I$ and $J$ have
the form $I = \Delta(A, X, \phi, Y, B)$
and $J = \Delta(B, 1, \psi, 1, C)$. Then
$t_I^{F, G} t_J^{G, H} = t_K^{F, H}$
where $K = \Delta(A, X, \phi \comp
\uu{\psi}, \psi^{-1}(Y), C)$ and
$\uu{\psi}$ is induced by $\psi$.
\end{lem}

\begin{proof}
Let $I$, $J$, $K$ be as in part (1). Then
$K = I * J$. We have $I = \Delta(\phi)$
and $I^\circ = \Delta(\phi^{-1})$. So
$I * I^\circ = \Delta(A)$ and
$I^\circ * I = \Delta(B)$. Hence,
$I^\circ * K = J$. Let $L \in \cS(F, H)$
be such that $\tau_L^{I, J} \neq 0$.
By part (2) of Theorem \ref{5.2},
$L \in \ad(K)$. The elements of
$\cP_L^{I, J}$ are the pairs $(I, V)$
where $V \leq J$ and $L \leq I * V$.
The condition on $V$ can be expressed
as $I^\circ * L \leq V \leq J$. By the
defining formula, $\tau_L^{I, J} =
\sum_V \mob(V, J)$. Since
$\tau_L^{I, J} \neq 0$, the recursive
characterization of the M\"{o}bius
function yields $I^\circ * L = J$, that
is, $L = K$. Moreover, $\tau_K^{I, J} =
\mob(J, J) = 1$. Part (1) is established.
Part (2) holds by a similar argument.
Alternatively part (2) follows from
part (1) by considering opposite
morphisms.
\end{proof}

\begin{lem}
\label{6.2}
Given $B \leq G \in \cK$, then
$s_{\Delta(B)}^{G, G} = \sum_{Y
\in \cS(B)} t_{\Delta(Y)}^{G, G}$ as
a sum of mutually orthogonal
idempotents of $\End_\Lambda(G)$.
\end{lem}

\begin{proof}
This is immediate from the previous
lemma.
\end{proof}

\begin{lem}
\label{6.3}
Given $A \leq F \in \cK \ni G \geq B$,
then $\{ s_U^{F, G} : U \in \cS(A, B) \}$
and $\{ t_I^{F, G} : I \in \cS(A, B) \}$ are
$\kay$-bases for $s_{\Delta(A)}^{F, F} .
\Lambda(F, G) . s_{\Delta(B)}^{G, G}$.
\end{lem}

\begin{proof}
By the previous two lemmas and
part (3) of Theorem \ref{5.2},
$s_{\Delta(A)}^{F, F} t_{I'}^{F, G}
s_{\Delta(B)}^{G, G} = 0$ for all
$I' \in \cS(F, G) - \cS(A, B)$. The
rest of the argument is easy.
\end{proof}

The next result supplies a technique
for investigating the condition that
$\Lambda$ is semisimple.

\begin{pro}
\label{6.4}
Let $\cL, \cM \subseteq \cK$, let
$k : \cL \leftarrow \cM$ be a function
and, for each $G \in \cM$, let
$\kappa_G : k(G) \leftarrow G$ be
a group monomorphism. Then there
is a corner embedding
$\Lambda_{\cL} \leftarrow
\Lambda_{\cM}$ given by
$s_{(\kappa_F \times \kappa_G)(U)}^{k(F),
k(G)} \mapsfrom s_U^{F, G}$ for
$F, G \in \cM$ and $U \in \cS(F, G)$.
\end{pro}

\begin{proof}
By the previous lemma, the specified
algebra monomorphism is a corner
embedding.
\end{proof}

\begin{cor}
\label{6.5}
Given $G \in \cK$, then the algebra
$\End_{\kay \cS}(G) = \kay \cS(G, G)$
is semisimple if and only if $G$ is
trivial. In particular, the category
algebra $\kay \cS_\cK$ is locally
semisimple if and only if every group
in $\cK$ is trivial.
\end{cor}

\begin{proof}
Putting $\ell = \id_{\NN - \{ 0 \} }$,
then $\Lambda = k \cS_\cK$. If $G$
is trivial then $\End_{\kay \cS}(G)
\cong \kay$. Suppose $G$ is
non-trivial and let $A$ be a prime
order subgroup of $G$. We may
assume that $\cK = \{ A, G \}$.
Putting $\cL = \{ G \}$ and $\cM =
\{ A \}$, the latest proposition implies
that $\End_{\kay S}(A)$ is isomorphic
to a corner subalgebra of
$\End_{\kay \cS}(G)$. In Example
\ref{3.7}, we saw that
$\End_{\kay \cS}(A)$ is not
semisimple.
\end{proof}

The next two corollaries have similar
proofs, which become easy after first
noting that we can reduce to the case
where $\cK$ is finite.

\begin{cor}
\label{6.6}
Suppose $\cK$ has an element $H$
such that every element of $\cK$ is
isomorphic to a subgroup of $H$.
Then $\Lambda$ is locally semisimple
if and only if $\End_\Lambda(H)$ is
semisimple.
\end{cor}

\begin{cor}
\label{6.7}
Suppose, for all $F, G \in \cK$, there
exists $H \in \cK$ such that $F$ and
$G$ are isomorphic to subgroups of
$H$. Then $\Lambda$ is locally
semisimple if and only if every object
of $\Lambda$ has a semisimple
endomorphism algebra.
\end{cor}

We do not know whether, in analogy
with [Rog19, 7.3], the hypothesis on
$\cK$ in the latest corollary can be
dropped.

\section{Semisimplicity and algebraic
independence}

We prove Theorem \ref{1.1}.

We define the {\bf dual} of a
$\Lambda$-module $M$ to be the
$\Lambda$-module $M^*$ such that,
given $F, G \in \cK$, then $M^*(G)$ is
the dual $\kay$-module of $M(G)$ and,
given $s \in \Lambda(F, G)$, then
the action $s^\circ : M^*(G) \leftarrow
M^*(F)$ is the adjoint of the action
$s : M(F) \leftarrow M(G)$. When $\cK$
is finite, the finite-dimensional
$\kay$-modules $M^*$ and $M$ can
be regarded as mutual duals and, for
any $s \in \Lambda$, the action of
$s^\circ$ on $M^*$ is the adjoint of
the action of $s$ on $M$. We also use
a superscript $*$ to indicate the usual
dual of a module of a group algebra
over $\kay$.

Theorem \ref{3.6} gives the next
lemma, and the subsequent lemma
follows immediately.

\begin{lem}
\label{7.1}
Given an $\cS$-seed $(E, W)$ for $\cK$,
then $S_{E, W}^* \cong S_{E, W^*}$.
\end{lem}

\begin{lem}
\label{7.2}
Given a finite group $E$, then the simple
$\Lambda$-modules having minimal group
$E$ are all projective if and only if they
are all injective.
\end{lem}

For finite groups $E$ and $L$, we write
$\epi(E, L)$ to denote the set of group
epimorphisms $E \leftarrow L$. For
$\phi \in \epi(E, L)$, we write
$\uu{\phi} : E \leftarrow L / \ker(\phi)$
for the isomorphism induced by $\phi$.
Identifying the codomain $E$ of
$\uu{\phi}$ with $E/1$, we define
$$\ltri(\phi) = \Delta(E, 1, \uu{\phi},
  \ker(\phi), L) \in \cS(E, L) \; , \;\;\;\;\;\;
  \rtri(\phi) = \Delta(L, \ker(\phi),
  \uu{\phi}^{-1}, 1, E) \in \cS(L, E) \; .$$
Now suppose $\epi(E, L) \neq
\emptyset$ and $L$ is isomorphic to
a subgroup of an element of $\cK$. We
define a square matrix $T_E^L$ over
$\kay$, with rows and columns
indexed by $\epi(E, L)$, such that,
given $\phi, \psi \in \epi(E, L)$ then
the $(\phi, \psi)$-entry of $T_E^L$ is
$$T_E^L(\phi, \psi) =
  \tau_{\Delta(E)}^{\ltri(\phi), \rtri(\psi)}$$
In other words, to calculate $T_E^L$, we
evaluate $t_{\ltri(\phi)}^{F, G}
t_{\rtri(\psi)}^{G, H}$ as a $\kay$-linear
combination of round basis elements,
whereupon $T_E^L(\phi, \psi)$ is the
coefficient of $t_{\Delta(E)}^{E, E}$. Note
that the matrix $T_E^L$ is well-defined
up to conjugation by permutation
matrices and, as such, $T_E^L$
depends only on the isomorphism
classes of $E$ and $L$.

\begin{thm}
\label{7.3}
Suppose the matrix $T_E^L$ is invertible
for all finite groups $L$ isomorphic to a
subgroup of an element of $\cK$ and
all factor groups $E$ isomorphic to a
quotient group of $L$. Then
$\Lambda$ is locally semisimple.
\end{thm}

\begin{proof}
By Remark \ref{2.3}, we may assume
that $\cK$ is closed under subquotients
up to isomorphism. (For those who
prefer to work with finite-dimensional
algebras, the same remark also allows 
us to assume that $\cK$ is finite, though
that step is not needed.) Let $(E, W)$
be an $\cS$-seed for $\cK$. It
suffices to show that the simple
$\Lambda$-module $S_{E, W}$ is
projective. By the closure property of
$\cK$, we can replace $E$ with any
isomorphic copy, so we may assume
that $E \in \cK$. Inductively, we may
also assume that, given any $\cS$-seed
$(E', W')$ such that $E'$ is a strict factor
group of $E$, then $S_{E', W'}$ is
projective. Hence, by Lemma \ref{7.2},
each $S_{E', W'}$ is injective. Write
$\cE = \End_\Lambda(E)$. As usual, we
regard $W$ as a simple $\cE$-module
annihilated by the ideal $\cE_<$. Let
$i$ be a primitive idempotent of $\cE$
such that $iW \neq 0$. Then $i$ is
still primitive as an idempotent of
$\Lambda$ and $i S_{E, W} \neq 0$.
So the indecomposable projective
$\Lambda$-module $P = \Lambda i$
is the projective cover of $S_{E, W}$. We
are to show that $P \cong S_{E, W}$. For
a contradiction, suppose the unique
maximal $\Lambda$-submodule $Q$
of $P$ is non-zero.

The $\cE$-module $P(E) = \cE i$
is the projective cover of $W$. Since
$\kay \Aut(E)$ is semisimple,
$\cE i / \cE_< \cong W$, that is,
$Q(E) \cong \cE_< . i$. But the
inductive assumption implies that,
for any $(E', W')$ as above,
$S_{E', W'}$ cannot be a factor module
of $Q$. Therefore, $Q(E) = 0$ and
$\cE_<$ annihilates $P(E)$.

Let $G$ be of minimal order such that
$Q(G) \neq 0$. Since $Q(G) \leq P(G)
= \Lambda(G, E) i$, there exists
$v \in \Lambda(G, E)$ such that
$vi \in Q(G) - \{ 0 \}$. Now $v$ is a
$\kay$-linear combination of elements
having the form $s_V^{G, E}$ with
$V \in \cS(G, E)$. Fixing $V$, then
$$s_V^{G, E} = s_V^{G, E}
  s_D^{E, E} / \ell(V_\bullet)$$
where $D = \Delta(V^\bullet, V_\bullet,
\id, V^\bullet, V_\bullet)$. If the thorax
$\Theta(V) \cong V^\bullet / V_\bullet
\cong \Theta(D)$ is smaller than $E$,
then $s_D^{E, E} \in \cE_<$, hence
$s_D^{E, E} i = 0$ and
$s_V^{D, D} i = 0$. On the other hand,
if $\Theta(V) \cong E$, then
$V_\bullet = 1$ and $s_V^{G, E}$ is
a $\kay$-linear combination of
elements $t_J^{G, E}$ where
$J_\bullet = 1$. For such $J$, if
$J^\bullet \neq E$, then $t_J^{G, E}$
is a $\kay$-linear combination of
elements $s_{V'}^{G, E}$ with
$\Theta(V')$ smaller than $E$,
whence each $s_{V'}^{G, E} i = 0$
and $t_J^{G, E} i = 0$. So we may
assume that $v$ is a $\kay$-linear
combination of elements having the
form $t_J^{G, E}$ where $J^\bullet/
J_\bullet = E / 1$. Given $B \leq G$,
then $t_{\Delta(B)}^{G, G} =
t_{\Delta(B)}^{G, B} t_{\Delta(B)}^{B, G}$.
By the minimality of $G$ and the
closure assumption on $\cK$, if
$B < G$ then $Q(B) = 0$, whence
$t_{\Delta(B)}^{B, G} vi = 0$ and
$t_{\Delta(B)}^{G, G} vi = 0$. So, by
Lemma \ref{6.3},
$vi = t_{\Delta(G)}^{G, G} vi$. Hence,
by part (3) of Theorem \ref{5.2}, we
may assume that $v$ is a $\kay$-linear
combination of elements
$t_J^{G, E}$ where $J^\bullet /
J_\bullet = E / 1$ and
${}^\bullet J = G$. In other words,
$$v \in \bigoplus_{\psi \in \epi(E, G)}
  \kay \, t_{\rtri(\psi)}^{G, E} \; .$$

Since $\{ t_I^{E, E} : I \in
\cS(E, E)_< \}$ is a $\kay$-basis for
$\cE_<$, any $w \in \cE$ can be
expressed uniquely as
$w = w_< + w_=$ where
$w_< \in \cE_<$ and
$$w_= = \sum_{\epsilon \in \Aut(E)}
  \partial_{\Delta(\epsilon)}(w)
  t_{\Delta(\epsilon)}^{E, E}$$
with each
$\partial_{\Delta(\epsilon)}(w) \in
\kay$. Every simple factor of the
$\Lambda$-module $\Lambda \cE_<$
is isomorphic to $S_{E', W'}$ for some
$(E', W')$ as above. Since none of the
simple factors of $Q$ have that
form, $Q \cap \Lambda \cE_< = 0$.
So $v i_< \not\in Q - \{ 0 \}$. But
$v i \in Q - \{ 0 \}$. Therefore,
$v i_= \neq 0$. We have
$$t_{\rtri(\psi)}^{G, E}
  t_{\Delta(\epsilon)}^{E, E} =
  t_{\rtri(\epsilon^{-1} \psi)}^{G, E}$$
by Lemma \ref{6.1}. So we can write
$$v i_= = \sum_{\psi \in \epi(E, G)}
  v(\psi) t_{\rtri(\psi)}^{G, E}$$
with each $v(\psi) \in \kay$. Note
that $v(\psi) \neq 0$ for some $\psi$.
Let $(u(\phi) : \phi \in \epi(E, G))$ be
any family of elements
$u(\phi) \in \kay$. Write
$$u = {\sum}_\phi u(\phi)
  t_{\ltri(\phi)}^{E, G} \; .$$
Since $vi \in Q(G)$ and $\Lambda(E, G)
Q(G) = Q(E) = 0$, we have $uvi = 0$.
So $uv i_= = - u v i_< \in \cE_<$.
Therefore,
$$0 = \partial_{\Delta(E)}(uv i_=) =
  \sum_{\phi, \psi \in \epi(E, G)}
  u(\phi) T_E^G(\phi, \psi) v(\psi) \; .$$
But the vector $(u(\phi) : \phi)$ is
arbitrary, the vector $(v(\psi) : \psi)$
is non-zero and the matrix $T_E^G$
is invertible. We have obtained a
contradiction, as required.
\end{proof}

To apply the theorem, we shall be
needing another lemma.

\begin{lem}
\label{7.4}
Let $A$, $B$, $C$ be subgroups of finite
groups $F$, $G$, $H$, respectively. Let
$\phi : A \leftarrow B$ and $\psi : C
\leftarrow B$ be group epimorphisms
and $\theta : A \leftarrow C$ a group
isomorphism. Write $I = \ltri(\phi)$,
$J = \rtri(\psi)$, $K = \Delta(\theta)$.
Suppose $K \leq I * J$. Then $(I, J)$ is
compatible and $K \in \ad(I * J)$. We
have $I * J = (\theta {\times} \psi)(B)$.
Letting $S$ run over those subgroups
of $B$ such that $K \leq (\phi {\times}
\psi)(S)$, then
$$\tau_K^{I, J} = \sum_S \mob(S, B)
  \ell(\ker(\phi) \cap S \cap
  \ker(\psi)) \; .$$
Furthermore, if $\ker(\phi) =
\ker(\psi)$, then $K = I * J$ and
$t_I^{F, G} t_J^{G, H} = \tau_K^{I, J}
t_K^{F, H}$.
\end{lem}

\begin{proof}
The preambulatory parts and the rider
are clear. We need only prove the
formula for $\tau_K^{F, G}$. We shall
apply part (4) of Theorem \ref{5.2}.
Write $M = \ker(\phi)$ and
$N = \ker(\psi)$. Let $(R, T)$ run over
the pairs of subgroups $R \leq B \geq T$
such that $K \leq (\phi {\times} \psi)
(R \cap T)$. The discussion of the star
product in Section 2 shows that
$\cP_K^{I, J}$ is the set of pairs having
the form $(U_R, V_T)$, where
$$U_R = \ltri(\phi_R) =
  \Delta(A, 1, \uu{\phi_R}, M \cap R, R)
  \; , \dozspace V_T =
  \rtri(\psi_T) = \Delta(T, T \cap N,
  \uu{\psi_T}^{-1}, 1, C)$$
with $\uu{\phi_R}$ denoting the
isomorphism induced by the restriction
$\phi_R : A \leftarrow R$ of $\phi$,
similarly for $\uu{\psi_T}$ and
$\psi_T : C \leftarrow T$. The retraction
$\rho_K^{I, J}$ in the proof of
Theorem \ref{5.2} is given by
$$\rho_K^{I, J}(U_R, V_T) = ((\Delta(A,
  1, \phi_S, M \cap S, S), \Delta(S, S
  \cap N, \uu{\psi_S}^{-1}, 1, C))$$
where $S = R \cap T$. The pairs
having the form on the left-hand side
are precisely the elements of the
image $\cR_K^{I, J}$ of $\rho_K^{I, J}$,
moreover, the subgroups $S$ arising
in this way are precisely the subgroups
$S$ specified in the assertion.
\end{proof}

Suppose $\ell$ is algebraically
independent with respect to $\cK$.
Let $\Pi_\cK$ be the set of prime
divisors of the orders of the elements
of $\cK$. For each $q \in \Pi_\cK$,
write $\lambda_q = \ell(q)$. The
assumption on $\ell$ is that the
elements $\lambda_q \in \kay$ are
algebraically independent. Let $\cO$
be the integral domain generated over
$\QQ$ by the $\lambda_q$. Any
$\frako \in \cO$ can be expressed
uniquely as a polynomial expression
in $(\lambda_q : q \in \Pi_\cK)$
with coefficents in $\QQ$, so
we may speak of the {\bf degree} of
$\frako$. We call $\frako$ {\bf monic}
provided the associated polynomial
expression has a unique term of
maximal degree and the coefficient
for that term is $1$. For a positive
integer $n$, we write $\len(n)$ for the
length of $n$, we mean, the number
of prime factors up to multiplicity.

We now prove Theorem \ref{1.1}.
It suffices to show that, in the notation
of Theorem \ref{7.3}, the matrix $T_E^L$
is invertible. We shall show that, in fact,
the determinant of $T_E^L$ is a monic
element of $\cO$ with degree
$d^{|\epi(E, L)|}$, where
$d = \len(|L|/|E|)$. As before, we may
assume that $E, L \in \cK$. Let
$\phi \in \epi(E, L)$. By Lemma \ref{7.4},
$$T_E^L(\phi, \phi) = \tau_{\Delta
  (E)}^{\ltri(\phi), \rtri(\phi)} = \sum_{S
  \in \cS(B) : \Delta(E) \leq \phi(S)}
  \mob(S, B) \ell(\ker(\phi) \cap S)$$
which is a monic element of $\cO$
with degree $\len(|\ker(\phi)|) = d$.
Let $\psi \in \epi(E, L) - \{ \phi \}$.
It remains only to show that the
off-diagonal matrix entry
$$T_E^L(\phi, \psi) = \tau_{\Delta
  (E)}^{\ltri(\phi), \rtri(\psi)}$$
is an element of $\cO$ with degree
less than $d$. We may assume that
$T_E^L(\phi, \psi) \neq 0$. Then,
by part (2) of Theorem \ref{5.2},
$\Delta(E) \leq \ltri(\phi) * \rtri(\psi)$.
Bearing in mind that $\phi \neq \psi$,
the rider of Lemma \ref{7.4} implies
that $\ker(\phi) \neq \ker(\psi)$.
By the formula in the lemma,
$T_E^L(\phi, \psi)$ has degree
at most $\len(|\ker(\phi) \cap
\ker(\psi)|) < d$. The proof of
Theorem \ref{1.1} is complete.

\section{The trivial module}

We define the {\bf trivial
$\Lambda$-module} to be the
isomorphically unique simple
$\Lambda$-module $S_{1, 1}$
whose minimal groups are trivial.
In this section, we give some
criteria for $S_{1, 1}$ to be
projective and injective.

The set $\cK$ is finite if and only if the
algebra $\Lambda$ is unital. When that
condition holds, $\Lambda$ is semisimple
if and only if all the block algebras of
$\Lambda$ are simple. For finite $\cK$,
we define the {\bf principal block} of
$\Lambda$ to be the block $b_{1, 1}$
of $\Lambda$ that fixes $S_{1, 1}$.
We call $\Lambda b_{1, 1}$ the
{\bf principal block algebra}.

\begin{lem}
\label{8.1}
The trivial $\Lambda$-module
$S_{1, 1}$ is projective if and only
if $S_{1, 1}$ is injective. When $\cK$
is finite, that condition holds if
and only if the principal block algebra
$\Lambda b_{1, 1}$ is simple.
\end{lem}

\begin{proof}
This follows from Lemma \ref{7.2}.
\end{proof}

In the context of the biset category,
Rognerud \cite[7.6]{Rog19} noted that,
for finite $\cK$, an analogous simple
module of $\kay B_\cK$ (again called the
trivial module) is projective if and only
if $\kay B_\cK$ is semisimple. That
raises the question as to whether a
similar assertion holds for $\Lambda$.

Regarding $\ell$ as a formal function on
finite groups, let $\varphi$ denote the totient
functon for $\ell$. That is to say, $\varphi$
is the isomorphism invariant formal function
on finite groups determined by the
equivalent conditions
$$\ell(U) = \sum_{I \in \cS(U)} \varphi(I)
  \; , \dozspace \varphi(I) = \sum_{U \in
  \cS(I)} \mob(U, I) \ell(U)$$
where $U$ is a given finite group in the
first equation, $I$ likewise in the second.
The next result, an observation made by
Hall [Hal36], can be seen straight away by
considering the subgroup generated by
a given $d$-tuple of elements of a given
finite group.

\begin{pro}
\label{8.2} {\rm (Philip Hall.)}
Let $d$ be a positive integer. Suppose
$\ell(n) = n^d$ for all positive integers
$n$. Given a finite group $G$, then
$\varphi(G)$ is the number of
$d$-tuples $(g_1, ..., g_d)$ such that
$\{ g_1, ..., g_d \}$ is a generating set
for $G$. In particular, $\varphi(G)$ is
non-zero if and only if the minimal
number of generators of $G$ is at
most $d$.
\end{pro}

When $d = 1$, we have
$\varphi(G) = 0$ unless $G$ is cyclic,
in which case, $\varphi(G) = \phi(|G|)$,
where $\phi$ is the Euler totient function.

We present the next result as a separate
application of Lemma \ref{5.1}, but it
can also be seen as a specialization of
material implicit in the proof of
Boltje-Danz \cite[4.2]{BD13}.

\begin{lem}
\label{8.3}
Given a finite group $G$, then
$$\varphi(G) = \sum_{M, N \in \cS(G)}
  \mob(M, G) \mob(N, G)
  \ell(M \cap N) \; .$$
\end{lem}

\begin{proof}
Consider the poset
$\cP = \cS(G) {\times} \cS(G)$. Let
$\ell_\cP : \ZZ \leftarrow \cP$ be the
function given by $\ell_\cP(M, N) =
\ell(M \cap N)$ for $(M, N) \in \cP$.
The retraction $(M \cap N, M \cap N)
\mapsfrom (M, N)$ of $\cP$
preserves $\ell_\cP$ and has image
$\cP' = \{ (L, L) : L \in \cS(G) \}$.
Since $\mob_{\cP'}((K, K), (L, L)) =
\mob(K, L)$ for $K, L \leq G$, the
required conclusion follows from
Lemma \ref{5.1}.
\end{proof}

\begin{lem}
\label{8.4}
Let $F$, $G$, $H$ be finite groups and
$B \leq G$. Then $t_{1 \times B}^{F, G}
t_{B \times 1}^{G, H} = \varphi(B)
t_{1 \times 1}^{F, H}$.
\end{lem}

\begin{proof}
This follows from Lemmas
\ref{7.4} and \ref{8.3}.
\end{proof}

For arbitrary $\cK$, let $P_{1, 1}$
denote the projective cover of
$S_{1, 1}$. Given $G \in \cK$, let
$i_G = s_{1 \times 1}^{G \times G}
= t_{1 \times 1}^{G \times G}$. Using
Theorem \ref{5.2}, it is easy to show
that $i_G$ is primitive as an idempotent
of $\End_\Lambda(G)$ and hence also
as an idempotent of $\Lambda$. So
$\Lambda i_G$ is an indecomposable
projective $\Lambda$-module.

\begin{lem}
\label{8.5}
With the notation above, $\Lambda i_G
\cong P_{1, 1}$ as $\Lambda$-modules.
\end{lem}

\begin{proof}
By Proposition \ref{2.1}, we may
assume that $\cK$ owns a trivial
group $1$. Also writing $1$ to denote
the trivial subgroup of $G$ then, by
considering the elements
$s_{1 \times 1}^{1 \times G}$ and
$s_{1 \times 1}^{G \times 1}$, we
see that the idempotents $i_G$ and
$i_1$ are associate. Therefore,
$i_G S_{1, 1} \neq 0$.
\end{proof}

\begin{thm}
\label{8.6}
Suppose $\varphi(B) \neq 0$ for every
subgroup $B$ of every element of
$\cK$. Then $S_{1, 1}$ is projective
and injective, furthermore, if $\cK$
is finite, then the algebra
$\Lambda b_{1, 1}$ is simple.
\end{thm}

\begin{proof}
By Lemma \ref{8.1}, it suffices to
show that $P_{1, 1}$ is simple. Let $Q$
be a non-zero submodule of $P_{1, 1}$.
Let $G \in \cK$ such that $Q(G) \neq 0$.
Lemma \ref{8.5} allows us to put
$P_{1, 1} = \Lambda i_G$. If we can show
that $i_G \in Q(G)$, then it will follow that
$Q = P_{1, 1}$, and the simplicity of
$P_{1, 1}$ will be established. Let
$y \in Q(G) - \{ 0 \}$. By part (3) of
Theorem \ref{5.2},
$$y = \sum_{B \in \cS(G)}
  y_B t_{B, 1}^{G, G}$$
with each $y_B \in \kay$. Choose
$B$ such that $y_B \neq 0$. By
Lemma \ref{8.4}, $t_{1, B}^{G, G} y
= y_B \varphi(B) i_G$ Therefore,
$i_G \in Q(G)$, as required.
\end{proof}

The theorem together with
Proposition \ref{8.2} yields a
corollary.

\begin{cor}
\label{8.7}
Let $d$ be a positive integer. Suppose
$\ell(n) = n^d$ for all positive integers
$n$. Suppose also that every subgroup
of every element of $\cK$ has minimal
number of generators at most $d$.
Then the conclusion of Theorem
\ref{8.6} holds.
\end{cor}

\section{Deformation of the biset category}

We shall introduce a small $\kay$-linear
category $\Gamma$ on $\cK$, in other
words, an algebra $\Gamma$ over $\kay$
equipped with a complete family of
mutually orthogonal idempotents
$(\id_G^\Gamma : G \in \cK)$. Again,
$\Gamma$ will be determined by $\cK$,
$\kay$, $\ell$. A case of motivating
importance is that where $\ell$ is the
inclusion of the positive integers in
$\kay$, we mean to say, $\ell(n) = n$
for all positive integers $n$. In that
case, we shall find that $\Gamma$
coincides with $\kay \cB_\cK$, the
$\kay$-linear biset category on $\cK$.
That will justify interpretation of
$\Gamma$ as a deformation of
$\kay \cB_\cK$.

Let us briefly summarize the construction
of $\kay \cB_\cK$ and some other
related categories, starting with the
{\bf biset category} $\cB$. For a full
discussion, see Bouc
\cite[Chapters 2, 3]{Bou10}. The objects
of $\cB$ are the finite groups. Consider
finite groups $F$, $G$, $H$. An
{\bf $F$-$G$-biset} is defined to be an
$F {\times} G$-set with the action of $F$
written on the left, the action of $G$
written on the right. For an $F$-$G$-biset
$X$ and a $G$-$H$-biset $Y$, we write
$X \times_G Y$ to denote the set of
$G$-orbits of $X {\times} Y$, and we
regard $X \times_G Y$ as an
$F$-$H$-biset. The morphism set
$\cB(F, G)$ the Grothendieck group of
the category of $F$-$G$-bisets, with
relations such that addition is induced
by disjoint union. Thus, $\cB(F, G)$ has
a $\ZZ$-basis consisting of the
isomorphism classes $[(F {\times} G)/U]$
of the $F$-$G$-bisets having the
form $(F {\times} G)/U$, where
$U \in \cS(F, G)$. Note that
$[(F {\times} G)/U]$ is uniquely
determined by the conjugacy class
of $U$. The composition operation
$\cB(F, H) \leftarrow \cB(F, G) \times
\cB(G, H)$ is the $\ZZ$-linear map
induced by the operation $\times_G$
on bisets. Explicitly, \cite[2.3.34]{Bou10}
tells us that, for $U \in \cS(F, G)$ and
$V \in \cS(G, H)$, we have
$$\left[ \frac{F {\times} G}
  {U} \right] \left[ \frac{G {\times} H}
  {V} \right] = \sum_{U^\bullet . g .
  {}^\bullet V \subseteq G} \left[
  \frac{F {\times} H}
  {U * {}^{g \times 1} V} \right]$$
where the notation indicates that $g$
runs over representatives of the
double cosets of $U^\bullet$ and
${}^\bullet V$ in $G$.

We define the {\bf $\kay$-linear
biset category} $\kay \cB$ to be the
$\kay$-linear category whose objects
are the finite groups, the morphism
$\kay$-module $\kay \cB(F, G)$ is
the $\kay$-linear extension of
$\cB(F, G)$ and the composition for
$\kay \cB$ is the $\kay$-linear
extension of the composition for
$\cB$. Note that $\kay \cB$ is not
to be confused with the
$\kay$-linearization of $\cB$, which
could be expressed with the same
notation. As a direct sum of regular
$\kay$-modules,
$$\kay \cB(F, G) = \bigoplus_{U
  \in_{F \times G} \cS(F, G)} \kay
  \left[ \frac{F {\times} G} {U} \right]$$
with the notation indicating that $U$
runs over representatives of the conjugacy
classes of subgroups of $F {\times} G$.

We introduce a $\kay$-linear category
$\kay_\sigma \cB$, defined as follows.
The objects of $\kay_\sigma \cB$ are
the finite groups. The morphism
$\kay$-module $\kay_\sigma
\cB(F, G)$ has a formal $\kay$-basis
$\{ d_U^{F, G} : U \in_{F \times G}
\cS(F, G) \}$. We define
the composition to be such that
$$d_U^{F, G} d_V^{G, U} =
  \sum_{U^\bullet . g . {}^\bullet V
  \subseteq G} \frac{\ell(U_\bullet \cap
  {}^g ({}_\bullet V))}{|U_\bullet \cap
  {}^g ({}_\bullet V)|} \,
  d_{U * {}^{g \times 1} V}^{F, H} \; .$$
Some properties need to be checked.
It is not hard to see that the right-hand
expression is well-defined in that the
value does not change when $U$ and
$V$ are replaced by conjugate subgroups
or when a different choice is made for
the double coset representatives $g$.
We postpone, to Theorem \ref{9.3},
proof of the associativity of the
composition for $\kay_\sigma \cB$.
It is clear already that, if associativity
does hold, then $\kay_\sigma \cB$
is a $\kay$-linear category whose
identity morphism on $G$ is
$d_{\Delta(G)}^{G, G}$. It is also
clear already that, in the motivating
case mentioned above, where $\ell$
is the inclusion of the positive integers,
$\kay_\sigma \cB$ is indeed a
$\kay$-linear category and, in fact,
$\kay_\sigma \cB$ can be identified
with $\kay \cB$ by identifying each
$d_U^{F, G}$ with $[(F {\times} G)/U]$.

Now let us discuss the category
$\kay_\sigma \cS$. We make
$\kay_\sigma \cS(F, G)$ become
an $\FF(F {\times} G)$-module
via the conjuation action of
$F {\times} G$ on $\cS(F, G)$.
That is to say,
$${}^{f \times g} s_U^{F, G} =
  s_{{}^{f \times g} U}^{F, G} \; .$$
To describe the action in another
way, we introduce the unity-preserving
algebra map $\sigma_G : \kay_\sigma
\cS(G, G) \leftarrow \kay G$
given by $\sigma_G(g) =
s_{\Delta(G, g, G)}^{G, G}$
where $\Delta(G, g, G) =
\{ {}^g b \times b : b \in G \}$.

\begin{rem}
\label{9.1}
With the notation above, given
$x \in \kay_\sigma \cS(F, G)$,
then ${}^{f \times g} x =
\sigma_F(f) . x . \sigma_G(g^{-1})$.
\end{rem}

\begin{proof}
By $\kay$-linearity, we may assume
that $x = s_U^{F, G}$, whereupon
the verification is an easy calculation.
\end{proof}

Since $\kay G$ is semisimple, the
principal block of $\kay G$ is
$e_G = \sum_{g \in G} g / |G|$.
We define
$$\oo{s}_U^{F, G} = \sigma_F(e_F)
  . s_U^{F, G} . \sigma_G(e_G) =
  \frac{1}{|F| . |G|} \sum_{f \in F, g \in
  G} s_{{}^{f \times g} U}^{F, G} \; .$$
Plainly, $\oo{s}_U^{F, G}$ depends
only on the $F {\times} G$-conjugacy
class of $U$. We define a category
$\oo{\kay_\sigma \cS}$ such that
the objects are the finite groups and
the $\kay$-module of morphisms
$F \leftarrow G$ is the
$F {\times} G$-fixed submodule
$$\oo{\kay_\sigma \cS}(F, G) =
  (\kay_\sigma \cS(F, G))^{F
  \times G} = \!\!\! \bigoplus_{U
  \in_{F \times G} \cS(F, G)} \!\!\!
  \kay \oo{s}_U^{F, G} \; .$$
It is easy to check that
$\oo{\kay_\sigma \cS}$ is a
$\kay$-linear category whose
identity morphism on $G$ is the
element $\sigma_G(e_G) =
\oo{s}_{\Delta(G)}^{G, G}$. The
next theorem describes how
$\kay_\sigma \cB$ can be identified
with $\oo{\kay_\sigma \cS}$. First,
we need a lemma.

\begin{lem}
\label{9.2}
Let $F$, $G$, $H$ be finite groups.
Let $U \in \cS(F, G)$ and
$V \in \cS(G, H)$. Then
$$|U| . |V| = |U^\bullet . {}^\bullet V|
  . | U_\bullet \cap {}_\bullet V| .
  |U * V| \; .$$
\end{lem}

\begin{proof}
Let $A = \{ f {\times} g {\times} h :
f {\times} g \in U, g {\times} h \in V \}$.
Note that $U^\bullet \cap {}^\bullet V
= \{ g : f {\times} g {\times} h \in A \}$.
Fix $f {\times} g {\times} h \in A$. Given
$f' \in F$ and $h' \in H$, then
$f' {\times} g {\times} h' \in A$ if and
only if $f' f^{-1} \in {}_\bullet U$ and
$h' h^{-1} \in V$. So
$$|A| = |U^\bullet \cap {}^\bullet V| .
  |{}_\bullet U| . |V_\bullet| \; .$$
Given $g' \in G$, then $f {\times} g'
{\times} h \in A$ if and only if
$g' g^{-1} \in U_\bullet \cap
{}_\bullet V$. So
$$|A| = |U_\bullet \cap {}_\bullet V|
  . |U * V| \; .$$
Eliminating $|A|$ and using $|U| =
|U^\bullet| . |{}_\bullet U|$ and
$|V| = |{}^\bullet V| . |V_\bullet|$,
we obtain
$$
|U| . |V| . |U^\bullet \cap
  {}^\bullet V| = |U^\bullet| .
  |{}^\bullet V| . |U_\bullet \cap
  {}_\bullet V| . |U * V| \; .$$

\vspace{-0.3in}
\end{proof}

We let $\nu^{F, G} : \oo{\kay_\sigma
\cS}(F, G) \leftarrow \kay_\sigma
\cB(F, G)$ be the $\kay$-linear
isomorphism given by
$$\nu^{F, G}(d_U^{F, G}) =
  |G| \, \oo{s}_U^{F, G} / |U| \; .$$
Let us mention that another way
of expressing the formula is
$\nu^{F, G}(d_U^{F, G}) =
\tr_U^{F, G}(s_U^{F, G}) / |F|$,
where $\tr_U^{F, G}$ denotes the
transfer map from the $U$-fixed
submodule.

\begin{thm}
\label{9.3}
The composition operation on
$\kay_\sigma \cB$ is associative,
$\kay_\sigma \cB$ is a $\kay$-linear
category and the maps $\nu^{F, G}$
determine an isomorphism of
$\kay$-linear categories $\nu :
\oo{\kay_\sigma \cS} \leftarrow
\kay_\sigma \cB$ that acts on
objects as the identity.
\end{thm}

\begin{proof}
Letting $F$, $G$, $H$, $U$, $V$ be
as above, we must show that
$$\nu^{F, G}(d_U^{F, G}) \, \nu^{G, H}
  (d_V^{G, H}) = \nu^{F, H}
  (d_U^{F, G} d_V^{G, H}) \; .$$
Applying Remark \ref{9.1} to evaluate
$\sigma_G(e_G) s_V^{G, H}$, we obtain
$$s_U^{F, G} \sigma_G(e_G) s_V^{G, H}
  = \frac{1}{|G|} \sum_{g \in G}
  \sigma(U, {}^{g \times 1} V) \,
  s_{U * {}^{g \times 1} V}^{F, H} \; .$$
Since $\sigma(U, {}^{g \times 1} V) =
\ell(U_\bullet \cap {}^g ({}_\bullet V))$,
which depends only on
$U^\bullet . g . {}^\bullet V$, we have
$$\oo{s}_U^{F, G} \, \oo{s}_V^{G, H}
  = \frac{1}{|G|} \sum_{U^\bullet . g .
  {}^\bullet V \subseteq G}
  |U^\bullet . g . {}^\bullet V| \,
  \ell(U_\bullet \cap {}^g
  ({}_\bullet V)) \, \oo{s}_{U *
  {}^{g \times 1} V}^{F, H} \; .$$
Lemma \ref{9.2} now yields
$$\frac{\oo{s}_U^{F, G} \,
  \oo{s}_V^{G, H}}{|U| . |V|}
  = \frac{1}{|G|} \sum_{U^\bullet . g .
  {}^\bullet V \subseteq G}
  \frac{\ell(U_\bullet \cap
  {}^g ({}_\bullet V))}{|U_\bullet \cap
  {}^g ({}_\bullet V)|} \, . \,
  \frac{\oo{s}_{U * {}^{g \times 1}
  V}^{F, H}}{|U * {}^{g \times 1} V|} \; .$$
Multiplying by $|G| . |H|$ gives the
required equality.
\end{proof}

The isomorphism $\nu$ does not
preserve the antiisomorphisms induced
by taking opposite subgroups. We
mention that, if $\kay$ owns square
roots of the orders of all the groups in
$\cK$, then $\nu$ can be replaced by
an isomorphism, described as follows,
which does have that symmetry property.
For each $G$, we arbitrarily choose a
square root $\sqrt{|G|}$ of $|G|$. We
then replace $\nu^{F, G}$ with the map
$\sqrt{|F| . |G|} \, \oo{s}_U^{F, G} / |U|
\mapsfrom d_U^{F, G}$.

We now pass to a full subcategory. We
define
$$\Gamma = \kay_\sigma \cB_\cK$$
as a small $\kay$-linear category
and as an algebra over $\kay$. By a
comment earlier in this section,
when $\ell$ is the inclusion of the
positive integers, we can identify
$\Gamma = \kay \cB_\cK$ by
identifying $d_U^{F, G} =
[(F {\times} G)/U]$ for each $F \in
\cK \ni G$ and $U \in \cS(F, G)$.
The latest theorem tells us that
$\Gamma$ is isomorphic to a corner
subalgebra of $\Lambda$.
Corollary \ref{1.3} follows.

The same theorem also shows that, if
every group in $\cK$ is abelian, then
$\Gamma \cong \Lambda$. Now
suppose that every group in $\cK$ is
cyclic. To complete the proof of
Theorem \ref{1.4}, we must show
that $\Lambda$ is locally semisimple.
Let $E$ and $L$ be cyclic groups such
that $|E|$ divides $|L|$ which in turn
divides the order of an element of
$\cK$. By Theorem 7.3, it suffices to
show that the matrix $T_E^L$ is
invertible. Let $\phi, \psi \in \epi(E, L)$.
Let $M$ be the unique subgroup of $L$
such that $|M| = |L| / |E|$. Then
$\ker(\phi) = M = \ker(\psi)$. By
Lemma \ref{7.4}
$$t_{\medtriangleleft(\phi)}^{E, L}
  t_{\medtriangleright(\psi)}^{L, E} =
  \tau_{\Delta(\theta)}^{E, E}
  t_{\Delta(\theta)}^{E, E}$$
where $\theta = \phi \comp \psi^{-1}$.
If $\phi = \psi$, then $\Delta(\theta)
= \Delta(E)$ and
$$T_E^L(\phi, \psi) =
  \tau_{\Delta(E)}^{E, E} =
  \ell(M) \neq 0 \; .$$
If $\phi \neq \psi$, then
$\Delta(\theta) \neq \Delta(E)$ and
$T(\phi, \psi) = 0$. So $T_E^L$ is a
non-zero multiple of the identity
matrix. The proof of Theorem
\ref{1.4} is complete.

\end{document}